\documentclass[a4paper,12pt]{amsart}

\usepackage[margin=2.5cm]{geometry} 
\usepackage{amsthm,amsmath,amssymb}
\usepackage{enumitem}
\usepackage{hyperref}
\hypersetup{
    colorlinks = true, 
    urlcolor   = black,
    linkcolor  = black, 
    citecolor  = black 
}
\usepackage{tikz,tikz-cd}

\usepackage{verbatim}

\theoremstyle{plain}
\newtheorem{theorem}{Theorem}[section]
\newtheorem{prop}[theorem]{Proposition}
\newtheorem{cor}[theorem]{Corollary}
\newtheorem{lemma}[theorem]{Lemma}

\theoremstyle{definition}

\newtheorem{ex}[theorem]{Example}
\newtheorem{remark}[theorem]{Remark}

\newcommand{\C}[0]{\mathbb C}
\renewcommand{\d}[0]{\mathbf d}
\newcommand{\E}[0]{\mathcal E}

\newcommand{\lbeta}[0]{\bar{\beta}}
\newcommand{\N}[0]{\mathbb N}
\renewcommand{\P}{\mathbb P}
\newcommand{\Q}[0]{\mathbb Q}
\newcommand{\q}{\Q\text{HS}}
\newcommand{\x}[0]{\mathbf x}
\newcommand{\Z}[0]{\mathbb Z}
\newcommand{\z}{\Z\text{HS}}
\DeclareMathOperator{\lcm}{lcm}

\title[Normal surface singularities with integral homology sphere link]
{Normal surface singularities with an integral homology sphere link related to space monomial curves with a plane semigroup}

\author[J.~Mart\'{\i}n-Morales]{Jorge Mart\'{\i}n-Morales}
\address[J.~Mart\'{\i}n-Morales]{Centro Universitario de la Defensa, IUMA \\
Academia General Militar \\ 
Ctra.~de Huesca s/n. \\ 
50090 Zaragoza, Spain}
\email{jorge@unizar.es}

\author[L.~Vos]{Lena Vos}
\address[L. Vos]{
KU Leuven \\
Departement Wiskunde \\
Celestijnenlaan 200B, bus 2400 \\
3001 Leuven, Belgium}
\email{lena.vos@kuleuven.be}

\thanks{This project is partially supported by the Research Foundation - Flanders (FWO) project G.0792.18N. The first author is partially supported by MTM2016-76868-C2-2-P, Gobierno de Arag\'on (Grupo de referencia `\'Algebra y Geometr\'ia') cofunded by Feder 2014-2020 `Construyendo Europa desde Arag\'on', and by FQM-333 from `Junta de Andaluc\'ia'. The second author is supported by a PhD Fellowship of the Research Foundation - Flanders (no. 71587).}

\begin{document}

\begin{abstract}
In this article, we consider an infinite family of normal surface singularities with an integral homology sphere link which is related to the family of space monomial curves with a plane semigroup. These monomial curves appear as the special fibers of equisingular families of curves whose generic fibers are a complex plane branch, and the related surface singularities appear in a proof of the monodromy conjecture for these curves. To investigate whether the link of a normal surface singularity is an integral homology sphere, one can use a characterization that depends on the determinant of the intersection matrix of a (partial) resolution. To study our family, we apply this characterization with a partial toric resolution of our singularities constructed as a sequence of weighted blow-ups. 
\end{abstract}

\footnote{\emph{2020 Mathematics Subject Classification.} Primary: 14J17; 
Secondary: 14E15, 
57K10. 
}
\footnote{\emph{Key words and phrases.} normal surface singularities, rational and integral homology sphere link, (partial) resolution of singularities.} 

\maketitle


\section*{Introduction} \label{sec:intro}

For $(S,0) \subset (\C^n,0)$ a germ of a normal surface singularity, the \emph{link} $L(S,0)$ is defined as the intersection of $S$ with a small closed ball centered at the origin in $\C^n$. The topology of $L(S,0)$ can provide interesting information about the singularity $(S,0)$. For example, Mumford~\cite{Mum} showed that $L(S,0)$ is simply connected if and only if $S$ is smooth at~$0$, and Neumann~\cite{NeuCal} showed that $L(S,0)$ determines and is determined by the dual graph of a good resolution of $(S,0)$ (it is a \emph{graph manifold} whose plumbing decorated graph is a dual resolution graph of $(S,0)$). In this article, we are interested in normal surface singularities whose link is a \emph{integral homology sphere} or $\z$, that is, whose link has the same integral homology as a three-dimensional sphere. More generally, we can consider normal surface singularities having a \emph{rational homology sphere} or $\q$ link. \\

More precisely, we will study the link of an infinite family of normal surface singularities $(S,0) \subset (\C^{g+1},0)$ with $g \geq 2$ related to the family of \emph{space monomial curves with a plane semigroup}. These monomial curves arise as the special fibers of certain equisingular deformations of irreducible plane curve singularities. Such deformations are classical objects in singularity theory, and have been studied and generalized by, among others, Teissier, Goldin, Gonz\'alez-P\'erez, and Tevelev. Recently, they have played an important role in the solution of Yano's conjecture~\cite{BlancoYano}, and the solution of Dimca-Greuel's conjecture for branches by Alberich-Carrami\~nana et al.~\cite{AABM}. The motivation of the present work is the question by N\'emethi whether weighted blow-ups can be used to find examples of normal surface singularities having a $\z$ link. N\'emethi has intensely studied surfaces singularities and, more specifically, surface singularities with a $\q$ link, see for instance~\cite{NN} or, for a more modern approach,~\cite{CLM},~\cite{LN}, and the references listed there. \linebreak We will also compare our family with two important families of surface singularities. First, we will take a look at the \emph{Brieskorn-Pham surface singularities} $\{x_1^{a_1} + x_2^{a_2} + x_3^{a_3} = 0\} \subset (\C^3,0)$, whose link is a $\z$ if the exponents $a_i \geq 2$ for $i = 1,2,3$ are pairwise coprime. Second, we will consider the \emph{singularities of splice type}, having a $\z$ link, introduced by Neumann and Wahl~\cite{NW1}. We will see that our surface singularities $(S,0) \subset (\C^{g+1},0)$ with a $\z$ link are always of splice type, but they are never Brieskorn-Pham if $g\geq 3$. \\

To construct a space monomial curve with a plane semigroup, we start with a germ $\mathcal C:=\{f=0\}\subset (\C^2,0)$ of a complex plane curve defined by an irreducible series ${f\in\C[[x_0,x_1]]}$ with $f(0) = 0$. Let \[\nu_{\mathcal C}: R:=\frac{\C[[x_0,x_1]]}{(f)} \longrightarrow \N: h \mapsto \dim_\C \frac{\C[[x_0,x_1]]}{(f,h)}\] be the associated valuation. The semigroup $\Gamma(\mathcal C):=\{\nu_{\mathcal C}(h) \mid h \in R\setminus \{0\}\}\subset \N$ is finitely generated, and we can identify a unique minimal system of generators $(\bar{\beta}_0,\ldots,\bar{\beta}_g)$ of $\Gamma(\mathcal C)$. Define $(Y,0) \subset (\C^{g+1},0)$ as the image of the monomial map $M:(\C,0) \rightarrow (\C^{g+1},0)$ given by $M(t)=(t^{\bar{\beta}_0},\ldots,t^{\bar{\beta}_g}).$ This is an irreducible curve with the `plane' semigroup $\Gamma(\mathcal C)$ as its semigroup, and it is the special fiber of a flat family $\eta:(\chi,0) \subset (\C^{g+1} \times \C,0)\rightarrow (\C,0)$ whose generic fiber is isomorphic to $\mathcal C$. We call $Y$ a \emph{space monomial curve}, and the explicit equations defining $Y$ in $\C^{g+1}$ are of the form
	\[\renewcommand{\arraystretch}{1.6}{\left\{ 
	\begin{array}{r c l l}
	x_1^{n_1} & - & x_0^{n_0}  &  = 0 \\
	x_2^{n_2}  &- & x_0^{b_{20}}x_1^{b_{21}} &= 0  \\
	& \vdots & &\\
	x_g^{n_g} &- & x_0^{b_{g0}}x_1^{b_{g1}}\cdots x_{g-1}^{b_{g(g-1)}} & = 0,\\
 	 \end{array}
	\right.}\]
where $n_i > 1$ and $b_{ij} \geq 0$ are integers that are defined in terms of $(\bar{\beta}_0,\ldots, \bar{\beta}_g)$. \\

The \emph{monodromy conjecture} for these space monomial curves $Y \subset \C^{g+1}$ with $g\geq 2$ is proven in~\cite{MVV1} together with~\cite{MVV2}; an overview of these two articles can be found in the short note~\cite{MMVV}. Roughly speaking, the monodromy conjecture for $Y \subset \C^{g+1}$ states that the poles of the \emph{motivic, or related, Igusa zeta function} of $Y$ induce \emph{monodromy eigenvalues} of $Y$. The computation of the motivic Igusa zeta function and its poles is the main subject of~\cite{MVV2}; the study of the monodromy eigenvalues and proof of the monodromy conjecture can be found in~\cite{MVV1}. A key ingredient in this proof is the consideration of the curve $Y$ as the Cartier divisor $\{f_i = 0\}$ for some $i \in\{1,\ldots, g\}$ on a \emph{generic embedding surface} $S \subset \C^{g+1}$ defined by
\begin{equation}	\label{eq:equations-S-intro}
		\renewcommand{\arraystretch}{1.4}{\left\{\begin{array}{c c l l l}
		f_1& + & \lambda_2f_2  &  = 0 \\
		f_2& + & \lambda_3f_3 &= 0  \\
		& \vdots & &\\
		f_{g-1} & +  & \lambda_gf_g & = 0.\\
    	\end{array}\right.}
\end{equation}
For \emph{generic} coefficients $(\lambda_2,\ldots, \lambda_g) \in (\C\setminus \{0\})^{g-1}$, the scheme $S := S(\lambda_2,\ldots, \lambda_g)$ is a normal complete intersection surface which is smooth outside the origin. In this article, we are interested in the link of these normal surface singularities $(S,0) \subset (\C^{g+1},0)$. \\
\newpage

For $g = 2$, one can easily see that $(S,0) \subset (\C^3,0)$ is a Brieskorn-Pham surface singularity with exponents $n_0,n_1$ and $n_2$. For $g \geq3$, we will show that $(S,0) \subset (\C^{g+1},0)$ is never Brieskorn-Pham by considering the \emph{rupture exceptional curves} in the \emph{minimal good resolution} of $(S,0)$, that is, the exceptional curves that are either non-rational or have valency at least $3$ (i.e., they have at least $3$ intersections with other exceptional curves). While a Brieskorn-Pham surface singularity has at most one rupture exceptional curve in its minimal good resolution, our surface singularities have at least $g-1$ rupture exceptional curves. To show this, we will make use of a \emph{good $\Q$-resolution} of $(S,0)$, see Proposition~\ref{prop:rupture-divisors}. This is a resolution in which the final ambient space can have abelian quotient singularities, and the exceptional locus is a normal crossing divisor on such a space. A good $\Q$-resolution can be obtained as a sequence of \emph{weighted blow-ups} and induces a good resolution by resolving the singularities of the final ambient space, which are \emph{Hirzebruch-Jung singularities}. The good $\Q$-resolution $\hat{\varphi}: \hat{S} \rightarrow S$ that we will consider consists of the first $g-1$ steps of the embedded $\Q$-resolution of $Y \subset S$ constructed in~\cite[Section 5]{MVV1}. In particular, the \emph{dual graph} of $\hat{\varphi}$ is a tree as in Figure~\ref{fig:dual-graph}. \\

Since we already know that $(S,0)$ for $g = 2$ has a $\z$ link if and only if its exponents $n_i$ are pairwise coprime, it remains to investigate when $(S,0) \subset (\C^{g+1},0)$ has a $\z$ link for $g\geq 3$. For this purpose, we will make use of a characterization for a general normal surface singularity $(S,0) \subset (\C^n,0)$ to have a $\z$ link depending on the \emph{determinant} of $(S,0)$. The \emph{determinant} of a partial or good resolution of a normal surface singularity $(S,0)$ is defined as the determinant of the intersection matrix of the resolution, that is, it is the determinant of the matrix $(E_i\cdot E_j)_{ 1 \leq i,j \leq r}$, where $E_1,\ldots, E_r$ are the exceptional curves of the resolution. Geometrically, the cokernel of the intersection matrix of a good resolution of $(S,0)$ is equal to the torsion part of $H_1(L(S,0),\Z)$. The \emph{determinant} $\det(S)$ of a normal surface singularity $(S,0)$ is the absolute value of the determinant of some good resolution of $(S,0)$. Since the torsion part of $H_1(L(S,0),\Z)$ is a finite group of order $\det(S)$, this is independent of the chosen good resolution. In practice, $\det(S)$ can be computed as the product of the absolute value of the determinant of a partial resolution of $(S,0)$ and the orders of the small groups acting on the singularities of the new ambient space, see~\eqref{eq:det-bir-morphism}. \\

 In terms of the determinant $\det(S)$ and a good resolution $\pi: \tilde S \rightarrow S$ of  a normal surface singularity $(S,0)$, the characterization can now be formulated as follows: the link $L(S,0)$ is a $\z$ if and only if $\det(S) = 1$ and $\pi$ has only rational exceptional curves and a tree as dual graph. Because the good $\Q$-resolution $\hat{\varphi}: \hat{S} \rightarrow S$ of our singularities has a tree as dual graph, and the singularities of $\hat{S}$ can be resolved with rational exceptional curves and a bamboo-shaped dual graph, we only need to check when both $\det(S) = 1$ and the exceptional curves of $\hat{\varphi}$ are rational. Furthermore, we can express $\det(S)$ in terms of the orders of the singularities of $\hat{S}$ and the determinant of $\hat{\varphi}$. To compute the latter determinant, we will first prove in Proposition~\ref{prop:detA} a formula for the determinant of a general good $\Q$-resolution with the same dual graph as in Figure~\ref{fig:dual-graph} by rewriting it in terms of a specific kind of tridiagonal matrices. For our good $\Q$-resolution $\hat{\varphi}: \hat{S} \rightarrow S$, this immediately implies the expression for the determinant of $\hat{\varphi}$ and of the singularity $(S,0)$ in Corollary~\ref{cor:detA} and Corollary~\ref{cor:detS}, respectively. Together with the properties of $\hat{\varphi}$, this yields the following theorem. As the same approach also gives conditions for $(S,0)$ to have a $\q$ link, which is true if and only if the dual graph of a good resolution is a tree with only rational exceptional curves, we can state our main result as the following generalization of the characterization for Brieskorn-Pham surface singularities. 
 
{\renewcommand\thetheorem{\Alph{theorem}}
\begin{theorem}\label{thm:rat-int-hom}
Let $(S,0) \subset (\C^{g+1},0)$ be a normal surface singularity defined by the equations~\eqref{eq:equations-S-intro} with $g \geq 2$. The link of $(S,0)$ is a $\q$ if and only if for all $k = 1,\ldots, g-1$, we have $\gcd(n_k,\lcm(n_{k+1},\ldots, n_g)) = 1$ or ${\gcd(\frac{\lbeta_k}{e_k},\lcm(n_{k+1},\ldots, n_g)) =1}$, where $e_k := \gcd(\lbeta_0,\ldots,\lbeta_k)$. The link of $(S,0)$ is a $\z$ if and only if the exponents $n_i$ for $i = 0,\ldots, g$ are pairwise coprime and $\gcd(\frac{\lbeta_k}{e_k},e_k) = 1$ for $k = 2,\ldots, g-1$. 	
\end{theorem}}

Once we have shown this result, it is easy to check that our surface singularities $(S,0) \subset (\C^{g+1},0)$ with $g\geq 2$ having a $\z$ link are of splice type. To this end, we will determine their \emph{splice diagram}. This is a finite tree in which every vertex has either valency~$1$, called a \emph{leaf}, or valency at least $3$, called a \emph{node}, and in which a weight is assigned to each edge starting at a node. Every dual graph of a normal surface singularity with a $\z$ link corresponds to a unique splice diagram of special type. Hence, such a splice diagram also determines and is determined by the link. Furthermore, if a splice diagram of a $\z$ link satisfies the so-called \emph{semigroup condition}, then Neumann and Wahl constructed in~\cite{NW1} an isolated complete intersection surface singularity having this link, called a \emph{singularity of splice type}. In addition, they conjectured that every normal complete intersection surface singularity with a $\z$ link is of splice type. To compute the splice diagram of our surface singularities having a $\z$ link, we will once more consider the good $\Q$-resolution $\hat{\varphi}: \hat{S} \rightarrow S$. We will see that the semigroup condition is fulfilled and that our surface singularities with a $\z$ link are always of splice type. In particular, they support the conjecture of Neumann and Wahl. \\

This article is organized as follows. We start in Section~\ref{sec:prelim} by briefly discussing the necessary background. In Section~\ref{sec:surfaces}, we will introduce our surface singularities $(S,0) \subset (\C^{g+1},0)$ in more detail, list the main properties of the considered good $\Q$-resolution of $(S,0)$, and use this resolution to show that $(S,0)$ is not Brieskorn-Pham for $g\geq 3$ and to show the conditions for its link to be a $\q$. In Section~\ref{sec:proof}, we will prove the characterization for $(S,0)$ to have a $\z$ link by computing its determinant, give some concrete examples in Example~\ref{ex:int-hom}, and show that our surface singularities with a $\z$ link are always of splice~type.\\

\noindent\textbf{Acknowledgement.} We would like to thank Andr\'as N\'emethi for  the suggestion to study normal surface singularities with an integral homology sphere link in the realm of weighted blow-ups, and we would like to thank Jonathan Wahl for pointing out the connection with the singularities of splice type. We would also like to thank Willem Veys, Enrique Artal Bartolo and Jos\'e~I.~Cogolludo-Agust\'in for initiating our collaboration.


\section{Preliminaries}\label{sec:prelim}

In this preliminary section, we give a short overview of the background needed in this article. We start by fixing some notation and conventions. First, by a \emph{(complex) variety}, we mean a reduced separated scheme of finite type over $\C$, which is not necessarily irreducible. A one-dimensional (resp. two-dimensional) variety is called a \emph{curve} (resp. \emph{surface}). Second, for a rational number $\frac{a}{b}$, we denote by $[\frac{a}{b}]$ its integer part. Third, for a set of integers $m_1,\ldots, m_r \in \Z$, we denote by $\gcd(m_1,\ldots, m_r)$ and $\lcm(m_1,\ldots, m_r)$ their greatest common divisor and lowest common multiple, respectively. To shorten the notation, we will sometimes use $(m_1,\ldots, m_r)$ for the greatest common divisor.  

\subsection{Space monomial curves with a plane semigroup}\label{sec:space-monomial}

Let $\mathcal C := {\{f = 0\} \subset (\C^2,0)}$ be an irreducible plane curve singularity defined by a complex series $f \in \C[[x_0,x_1]]$ with $f(0)=0$, and let \[\nu_{\mathcal C}: \frac{\C[[x_0,x_1]]}{(f)} \setminus \{0\} \longrightarrow \N: h \mapsto \dim_\C \frac{\C[[x_0,x_1]]}{(f,h)}\] be its associated valuation. The \emph{semigroup} $\Gamma(\mathcal C)$ is the image of this valuation and can be generated by a unique minimal system of generators $(\lbeta_0,\ldots,\lbeta_g)$ with $\lbeta_0 < \cdots < \lbeta_g$ and $\gcd(\lbeta_0, \ldots, \lbeta_g) = 1$. Furthermore, the sequence $(\lbeta_0,\ldots,\lbeta_g)$ determines and is determined by the topological type of $\mathcal C$, see for instance~\cite{Z}. Therefore, it is a natural question how one can recover the equation of a plane curve singularity from a given topological type. \\

In~\cite{T1}, Teissier provides a way to describe every plane curve singularity with given data $\Gamma(\mathcal C) = \langle \lbeta_0,\ldots,\lbeta_g \rangle$ as an equisingular deformation of the monomial curve $Y \subset (\C^{g+1},0)$ defined as the image of the monomial map $M:(\C,0) \rightarrow (\C^{g+1},0)$ given by $t \mapsto (t^{\bar{\beta}_0}, \ldots,t^{\bar{\beta}_g})$. This is an irreducible (germ of a) curve which has the `plane' semigroup $\Gamma(\mathcal C)$ as semigroup, which is smooth outside the origin, and which can be seen as a deformation of $\mathcal C$ in the following way. First, if we define the integers $e_i:=\gcd(\bar{\beta}_0,\ldots,\bar{\beta}_i)$ for ${i=0,\ldots,g}$ satisfying ${\lbeta_0 = e_0 > e_1 > \cdots > e_g = 1}$, and $n_i:=\frac{e_{i-1}}{e_i} \geq 2$ for $i=1,\ldots,g$, then $n_i\bar{\beta}_i$ for ${i=1,\ldots,g}$ is contained in the semigroup generated by $\lbeta_0,\ldots,\lbeta_{i-1}$. Hence, there exist non-negative integers $b_{ij}$ for $0 \leq j < i$ such that \[n_i\bar{\beta}_i=b_{i0}\bar{\beta}_0+\cdots +b_{i(i-1)}\bar{\beta}_{i-1},\] and these integers are unique under the extra condition that $b_{ij}<n_j$ for $j \neq 0$. For simplicity, we put $n_0 : =b_{10}$, and we state the following useful properties that we will use later on: 
\begin{enumerate}
	\item[(i)] for $i = 0,\ldots, g-1$, we have that $e_i = n_{i+1}\cdots n_g$; 
	\item[(ii)] for $i = 0,\ldots, g-1$, we have that $n_j \mid \lbeta_i$ for all $j > i$;
	\item[(iii)] for $i = 1,\ldots, g$, we have that $\gcd(\frac{\lbeta_i}{e_i},n_i) = \gcd(\frac{\lbeta_i}{e_i},\frac{e_{i-1}}{e_i}) = 1$, and, in particular, that $\gcd(n_0,n_1) = \gcd(\frac{\lbeta_1}{e_1},n_1) = 1$; and
	\item[(iv)] for $i = 1,\ldots, g$, we have that $n_i\lbeta_i < \lbeta_{i+1}$.
\end{enumerate} 
Using a \emph{minimal generating sequence} of the valuation $\nu_{\mathcal C}$, one can construct a family $\eta:(\chi,0) \subset (\C^{g+1}\times \C,0) \rightarrow (\C,0)$ of germs of curves in $(\C^{g+1} \times \C,0)$, which is \emph{equisingular} for instance in the sense that $\Gamma(\mathcal C)$ is the semigroup of all curves in the family. The generic fiber $\eta^{-1}(v)$ for $v \neq 0$ is isomorphic to $\mathcal C$, and the special fiber $\eta^{-1}(0)$ is defined in $(\C^{g+1},0)$ by the equations $x_i^{n_i} - c_ix_0^{b_{i0}} \cdots x_{i-1}^{b_{i(i-1)}} = 0$ for $i = 1, \ldots, g$. The coefficients $c_i$ are needed to see that any irreducible plane curve singularity with semigroup $\Gamma(\mathcal C)$ is an equisingular deformation of such a monomial curve. However, for simplicity, we can assume that every $c_i = 1$, which is always possible after a suitable change of coordinates. This yields the monomial curve $Y$. \\

Clearly, we can also consider the global curve in $\C^{g+1}$ defined by the above binomial equations; from now on, we define \emph{a (space) monomial curve} $Y \subset \C^{g+1}$ as the complete intersection curve given by
	\begin{equation} \label{eq:equations-Y}
		\renewcommand{\arraystretch}{1.3}{\left\{\begin{array}{r c l l}
		f_1:= x_1^{n_1} & - & x_0^{n_0}  &  = 0 \\
		f_2:= x_2^{n_2}  &- & x_0^{b_{20}}x_1^{b_{21}} &= 0  \\
		& \vdots & &\\
		f_g := x_g^{n_g} &- & x_0^{b_{g0}}x_1^{b_{g1}}\cdots x_{g-1}^{b_{g(g-1)}} & = 0.\\
    	\end{array}\right.}
	\end{equation}
This is still an irreducible curve which is smooth outside the origin. In~\cite{MVV1}, the \emph{monodromy eigenvalues} for such a space monomial curve $Y \subset \C^{g+1}$ with $g\geq 2$ are investigated by considering $Y$ as a Cartier divisor a \emph{generic embedding surface} $S \subset \C^{g+1}$. Together with the results from~\cite{MVV2}, this yields a proof of the \emph{monodromy conjecture} for $Y \subset \C^{g+1}$. In this article, we are interested in the topology of these generic embedding surface singularities $(S,0) \subset (\C^{g+1},0)$. We will introduce them in detail in Section~\ref{sec:surfaces}.

\subsection{Link of a normal surface singularity}

Let $(S,0) \subset (\C^n,0)$ be a germ of a normal surface singularity. Its \emph{link} $L(S,0)$ is an oriented three-dimensional manifold which is defined as the intersection of~$S$ with a small enough closed ball centered at the origin in $\C^n$. In this article, we are interested in normal surface singularities whose link is a \emph{rational (resp. integral) homology sphere}, that is, whose link has the same rational (resp. integral) homology as a three-dimensional sphere. In this case, we will say that the link is a $\q$ (resp. a $\z$). To study when the link $L(S,0)$ is a $\q$ or a $\z$, we can make use of a practical criterion in terms of the determinant and a good resolution of~$(S,0)$. \\

By a \emph{good resolution} of $(S,0)$, we mean a proper birational morphism $\pi: \tilde S \rightarrow S$ from a smooth surface $\tilde S$ to $S$ which is an isomorphism over $S\setminus \{0\}$ and whose exceptional locus $\pi^{-1}(0)$ is a simple normal crossing divisor (i.e., its irreducible components, called the \emph{exceptional curves}, are smooth and intersect normally). It is well known that such a resolution always exists as a sequence of blow-ups at well-chosen points. A good resolution $\pi:\tilde S \rightarrow S$ is called \emph{minimal} if every other good resolution of $(S,0)$ factors through $\pi$. Equivalently, $\pi$ is minimal if there is no exceptional curve that can be contracted (by blowing down) so that the resulting morphism is still a good resolution of $(S,0)$. It is worth mentioning that, by Castelnuovo's Contractibility Theorem, the only possible exceptional curves that can be contracted in such a way are rational and have self-intersection number~$-1$. Furthermore, a minimal good resolution of a normal surface singularity $(S,0)$ always exists and is unique up to isomorphism. Therefore, we call it \emph{the minimal good resolution} of~$(S,0)$. \\

With a good resolution of $(S,0)$, we can associate a \emph{dual graph} $\Gamma$ whose vertices correspond to the exceptional curves $E_1,\ldots, E_r$, and two vertices $E_i$ and $E_j$ are connected by an edge if and only if $E_i \cap E_j \neq \emptyset$. Often, each vertex $E_i$ is labeled with two numbers $(g_i,-\kappa_i)$, where $g_i$ is the genus of $E_i$ and $-\kappa_i$ its self-intersection number. It is a classical result that the free part of $H_1(L(S,0),\Z)$ has rank $2\sum_{i=1}^rg_i + b,$ where $b$ is the number of loops in the dual graph, and that its torsion part is equal to $\text{coker}(A)$, where $A = (E_i \cdot E_j)_{ 1 \leq i,j \leq r}$ is the intersection matrix of the good resolution, see for example~\cite[Ch. 2, Prop. 3.4]{Dim}. In particular, as $A$ is negative definite, which was originally noted by DuVal but also shown by Mumford in~\cite{Mum}, the torsion part of $H_1(L(S,0),\Z)$ is a finite group of order $\vert \det(A) \vert = \det(-A)$. \\

This result has two immediate consequences. First, it implies that $\det(-A)$ is independent of the chosen good resolution of $(S,0)$. Hence, we can define the \emph{determinant} of $(S,0)$ as $\det(S) := \det(-A)$ with $A$ the intersection matrix of any good resolution of $(S,0)$. Second, we find the following easy conditions for $(S,0)$ to have a $\q$ or $\z$ link.

\begin{theorem}\label{thm:conditions-rat-int}
Let $(S,0) \subset (\C^n,0)$ be a normal surface singularity, and consider a good resolution $\pi: \tilde S \rightarrow S$. The link of $(S,0)$ is a $\q$ if and only if all exceptional curves of $\pi$ are rational and the dual graph $\Gamma$ of $\pi$ is a tree. The link of $(S,0)$ is a $\z$ if and only if it is a $\q$ and $\det(S) = 1$.
\end{theorem} 

To compute the determinant of $(S,0)$ in practice, we do not really need a good resolution of $(S,0)$: if $\pi:\tilde S \rightarrow S$ is a proper birational morphism from a normal surface $\tilde S$ to $S$ which is an isomorphism over $S \setminus \{0\}$, then
\begin{equation}\label{eq:det-bir-morphism}
	\det(S)  =  \det(-A) \prod_{p \in \pi^{-1}(0)} \det(\tilde S,p),
\end{equation}
see for instance~\cite[Lemma 4.7]{ACM}. Here, $A = (E_i\cdot E_j)_{1\leq i,j \leq r}$ is the intersection matrix of $\pi$, where $E_1,\ldots, E_r$ are the exceptional curves of $\pi$, and $\det(\tilde S,p)$ is the absolute value of the determinant of the intersection matrix of some good resolution at $p$. Note that if $p \in \pi^{-1}(0)$ is written as a \emph{Hirzebruch-Jung singularity of type} $\frac{1}{d}(1,q)$ with $d$ and $q$ coprime, then $\det(\tilde S,p) = d$, see Section~\ref{sec:quotient-sing}. 

\subsection{Brieskorn-Pham surface singularities}\label{sec:Brieskorn-Pham}

An important family of normal surface singularities whose link is a $\q$ or $\z$ are Brieskorn-Pham surface singularities \[S(a_1,a_2,a_3) := \{F_{(a_1,a_2,a_3)} = x_1^{a_1} + x_2^{a_2} + x_3^{a_3} = 0\} \subset (\C^3,0)\] satisfying some conditions in terms of the exponents $a_i\geq 2$. The most classical characterization uses a graph $G(a_1,a_2,a_3)$ associated with these exponents, see for example~\cite[Satz 1]{Bri} or \cite[Ch. 3, Thm. 4.10]{Dim}. In~\cite[Prop. 5.1]{ACM}, an equivalent characterization is obtained by considering Brieskorn-Pham surface singularities as a special case of \emph{weighted L\^e-Yomdin singularities}. More precisely, put $e := \gcd(a_1,a_2,a_3)$ and $\alpha_l := \frac{1}{e}\gcd(a_i,a_j)$ for every $\{i,j,l\} = \{1,2,3\}$. Then, $F_{(a_1,a_2,a_3)}$ is $\omega$-weighted homogeneous with \[\omega := \frac{1}{e^2\alpha_1\alpha_2\alpha_3}(a_2a_3,a_1a_3,a_1a_2),\] and $S(a_1,a_2,a_3)$ can be seen as an $(\omega,k)$-weighted L\^e-Yomdin singularity for any ${k\geq 1}$. Following the approach in~\cite[4.3]{ACM} for weighted L\^e-Yomdin singularities, we can consider the curve $C := \{x_1^{e\alpha_2\alpha_3} + x_2^{e\alpha_1\alpha_3} + x_3^{e\alpha_1\alpha_2} = 0\}$ in the \emph{weighted projective plane} $\P^2_{(\alpha_1,\alpha_2,\alpha_3)}$ (see Section~\ref{sec:quotient-sing}), which has genus \[\frac{e^2\alpha_1\alpha_2\alpha_3 - e(\alpha_1+\alpha_2+\alpha_3) + 2}{2}.\] Furthermore, the determinant of $S(a_1,a_2,a_3)$ is given by \[ed_1^{e\alpha_1-1}d_2^{e\alpha_2-1}d_3^{e\alpha_3-1},\] where $d_i := \frac{a_i}{e\alpha_j\alpha_l}$ for $\{i,j,l\} = \{1,2,3\}$. Now, $S(a_1,a_2,a_3)$ has a $\q$ (resp. $\z$) link if and only if the above genus is equal to $0$ (resp. and the determinant is equal to $1$). This yields the following result.

\begin{prop}\label{prop:Brieskorn-Pham-rat-int}
Using the above notations, the link of a Brieskorn-Pham surface singularity $S(a_1,a_2,a_3) \subset (\C^3,0)$ is a $\q$ if and only if either $\alpha_1 = \alpha_2 = \alpha_3 = 1$ and $e = 2$, or $\alpha_i = \alpha_j  = e = 1$ for some $i \neq j$. It is a $\z$ if and only if the exponents $a_1,a_2$ and $a_3$ are pairwise coprime.
\end{prop} 

\begin{remark}\label{rem:Brieskorn-Pham}
In fact, we do not really need the theory of weighted L\^e-Yomdin singularities and the results from~\cite[4.3]{ACM} to obtain this result. Alternatively, one could directly consider Theorem~\ref{thm:conditions-rat-int} with the partial resolution of $S(a_1,a_2,a_3)$ consisting of one weighted blow-up at the origin with weight vector $\omega$. This resolution has one exceptional curve $\E \subset \P^2_\omega$ which is isomorphic to the curve $C \subset \P^2_{(a_1,a_2,a_3)}$, and which contains three sets of singular points, corresponding to the coordinate axes in $\P^2_\omega$. This gives the same genus and determinant as above. For more details, see~\cite[Example 3.6]{Ma1}.
\end{remark}

\subsection{Singularities of splice type}\label{sec:splice-type}

Since we know how to recover all plane curve singularities of a given topological type, it is natural to ask whether this is also possible for surface singularities with a given link. Unfortunately, this question is still open, even if the link is a $\q$ or $\z$. Here, we restrict to briefly explaining some results and conjectures in the $\z$ case. For more details, see~\cite{NW1} and~\cite{NW2}. \\

With a $\z$ link, we can associate a unique \emph{splice diagram}, originally introduced by Siebenmann~\cite{Sie}. This is a finite tree in which every vertex has either valency~$1$, called a \emph{leaf}, or valency at least $3$, called a \emph{node}, and in which a weight is assigned to each edge starting at a node. In~\cite{EN}, Eisenbud and Neumann showed that the links of normal surface singularities that are a $\z$ are in one-one correspondence with splice diagrams satisfying the following~conditions: 

\begin{enumerate}
	\item [(i)] the weights around a node are positive and pairwise coprime;
	\item [(ii)] the weight on an edge connecting a node with a leaf is greater than $1$; and 
	\item [(iii)] all \emph{edge determinants} are positive. 
\end{enumerate}

Here, the edge determinant for an edge connecting two nodes is the product of the two weights on the edge minus the product of the weights adjacent to the edge (i.e., the other weights around the two nodes). \\

The dual graph $\Gamma$ of a normal surface singularity with a $\z$ link yields a unique splice diagram $\Delta$ as follows. First, we suppress all vertices with valency~$2$. Then, for each edge $e$ starting at some node $v$, its weight $d_{ve}$ is the absolute value of the determinant of the intersection matrix of the subgraph $\Gamma_{ve}$ of $\Gamma$ obtained from cutting at $v$ in the direction of $e$. The other way around, one can obtain the dual graph $\Gamma$, and, hence, the link, from the splice diagram $\Delta$ by splicing or plumbing. For the details of this construction, we refer to~\cite{EN} or, for an easier method, to~\cite{NW2}. \\

In~\cite{NW1}, Neumann and Wahl constructed for a $\z$ link whose splice diagram satisfies the so-called \emph{semigroup condition} an isolated complete intersection surface singularity $(S,0)$ with this link, called a \emph{complete intersection singularity of splice type}. Up to date, there are no known examples of normal complete intersection surface singularities with a $\z$ link whose splice diagram does not satisfy the semigroup condition, or that are not of splice type. Therefore, Neumann and Wahl conjectured that every normal complete intersection surface singularity with a $\z$ link is of splice type (in particular, its splice diagram satisfies the semigroup condition). Earlier, in~\cite{NW}, Neumann and Wahl already conjectured the same for every Gorenstein normal surface singularity with a $\z$ link, but Luengo-Velasco, Melle-Hern\'{a}ndez and N\'{e}methi~\cite{LMN} found counterexamples to this conjecture. Furthermore, even when the splice diagram satisfies the semigroup condition, there always exist plenty of other analytic types (probably not complete intersections) with the same link. \\

Let us take a brief look at this semigroup condition and how to write the equations of the associated singularity of splice type. Consider a splice diagram~$\Delta$. For any vertices $v$ and $w$ of $\Delta$, we define the \emph{linking number} $l_{vw}$ of $v$ and $w$ as the product of all weights adjacent to, but not on, the shortest path from $v$ to $w$, and the number $l'_{vw}$ as the same product in which we omit the weights around $v$ and $w$. Using this notation, $\Delta$ is said to satisfy the \emph{semigroup condition} if and only if for every node $v$ and edge $e$ starting at $v$, the weight $d_{ve}$ is contained in the semigroup $\langle l'_{vw} \mid w \text{ is a leaf of } \Delta \text{ in } \Delta_{ve}\rangle \subset \N,$ where $\Delta_{ve}$ is the subgraph of $\Delta$ cut off from $v$ by $e$. If $\Delta $ satisfies this semigroup condition, we can associate \emph{admissable monomials} with each node $v$ and an edge $e$ starting at $v$ as follows. Relate to each leaf $w$ of $\Delta$ a variable $z_w$ and give it $v$-weight $l_{vw}$. Because $\Delta$ satisfies the semigroup condition, we can find (possibly non-unique) integers $\alpha_{vw} \in \N$ such that \[d_{ve} = \sum_{w \text{ leaf in } \Delta_{ve}} \alpha_{vw}l'_{vw}.\] Then, an admissable monomial associated with $v$ and $e$ is any monomial \[\prod_{w \text{ leaf in } \Delta_{ve}}z_w^{\alpha_{vw}}\] Note that its $v$-weight is equal to the product of all weights around $v$, also denoted by $d_v$. If the node $v$ has valency $\delta_v$, then we choose for every edge $e$ at $v$ one admissable monomial $M_{ve}$ and we make a system of $\delta_v - 2$ linear equations of the form \[\sum_{e \text{ edge at }v} a_{ie}M_{ve}, \qquad i = 1,\ldots, \delta_v - 2,\] where the coefficients $a_{ie}$ are chosen such that all maximal minors of the matrix $(a_{ie})_{i,e}$ have full rank. Finally, in each equation, we can add higher order terms with respect to the weights $l_{vw}$. The total number of equations is equal to $n-2$, where $n$ is the number of leaves in $\Delta$, and these equations define an isolated complete intersection singularity in $\C^n$ of \emph{splice type}. If one does not allow higher order terms, it is said to be of \emph{strict splice type}.

\begin{ex}\label{ex:Brieskorn-compl-int}
Consider a splice diagram with a single node $v$ of valency $n$. In this case, the semigroup condition is trivially fulfilled: $l'_{vw} = 1$ for any leaf $w$ of $\Delta$. Hence, each edge $e_j$ for $j = 1,\ldots, n$ with weight $d_j$ corresponds to a unique admissable monomial $M_j = z_j^{d_j}$, and the equations of strict splice type are of the form \[\sum_{j=1}^{n} a_{ij}z_j^{d_j}, \qquad i = 1,\ldots, n-2,\] where all maximal minors of the matrix $(a_{ij})_{i,j}$ have full rank. Equations of this type define isolated \emph{Brieskorn-Pham complete intersections} in $\C^{n}$, which are  generalizations of the Brieskorn-Pham surface singularities. 
\end{ex}

For other examples of this construction, we refer to~\cite{NW1} and~\cite{NW2}. In Section~\ref{sec:our-surfaces-splice-type}, we will check the semigroup condition and write the equations of strict splice type for our normal surface singularities having a $\z$ link. In particular, we will see that they are of splice type. Hence, they support the conjecture of Neumann and Wahl. 

\subsection{Quotient singularities and $\Q$-resolutions}\label{sec:quotient-sing}

To determine the conditions under which our surface singularities $(S,0)$ have a  $\q$ or $\z$ link, we will make use of a \emph{good $\Q$-resolution} of $(S,0)$. Roughly speaking, this a resolution in which the final ambient space can have abelian quotient singularities, and the exceptional divisor must have normal crossings on such a variety. In this section, we give a short introduction to quotient singularities and $\Q$-resolutions. We also touch briefly on an intersection theory on surfaces with abelian quotient singularities. More details can be found in~\cite{AMO1} and~\cite{AMO2}. \\

Consider an abelian quotient space $\C^n/G$ for $G \subset GL(n,\C)$ a finite abelian group. If we write $G = \mu_{d_1} \times \cdots \times \mu_{d_r}$ as a product of finite cyclic groups, where $\mu_{d_i}$ is the group of the $d_i$th roots of unity, then there exists a matrix $A = (a_{ij})_{i,j} \in \mathbb Z^{r\times n}$ such that $\C^n/G$ is isomorphic to the quotient of $\C^n$ under the action $(\mu_{d_1} \times \cdots \times \mu_{d_r})\times \C^n \rightarrow \C^n$ defined by $((\xi_1,\ldots,\xi_r) , (x_1,\ldots, x_n))  \mapsto (\xi_{d_1}^{a_{11}} \cdots \xi_{d_r}^{a_{r1}}\, x_1,\, \ldots\, , \xi_{d_1}^{a_{1n}} \cdots \xi_{d_r}^{a_{rn}}\, x_n )$.
This is called the \emph{quotient space of type $(\d,A)$}, where $\d := (d_1,\ldots, d_r)$, and denoted by \[ X(\d; A) := X \left( \begin{array}{c|ccc} d_1 & a_{11} & \cdots & a_{1n}\\ \vdots & \vdots & \ddots & \vdots \\ d_r & a_{r1} & \cdots & a_{rn} \end{array} \right).\] Note that we can always consider the $i$th row of $A$ modulo $d_i$. The class of an element ${\x := (x_1,\ldots, x_n) \in \C^n}$ under such an action $(\d;A)$ is denoted by $[\x]_{(\d;A)}$, where we omit the subindex if there is no possible confusion. Every quotient space $X(\d;A)$ is a normal irreducible $n$-dimensional variety whose singular locus is of codimension at least two and is situated on the \emph{coordinate hyperplanes} $\{x_i = 0\}$ for $i = 1,\ldots, n$, which are the images of the coordinate hyperplanes $\{x_i = 0\}$ in $\C^n$ under the natural projection $\C^n \rightarrow X(\d;A)$. \\

If $n = 2$, then one can show that each quotient space $X(\d;A) = \C^2/G$ is \emph{cyclic}, that is, it is isomorphic to a quotient space of type $(d;a,b)$. A cyclic type $(d;a,b)$ is said to be \emph{normalized}, and the corresponding quotient space $X(d;a,b)$ is said to be \emph{written in a normalized form}, if and only if $\gcd(d,a) = \gcd(d,b) = 1$. If this is not the case, we can normalize $X(d;a,b)$ as follows. First, we can assume that $\gcd(d,a,b) = 1$ as $X(d;a,b)$ is isomorphic under the identity morphism to $X(\frac{d}{k};\frac{a}{k},\frac{b}{k})$ for any $k$ dividing $d,a$ and $b$. Second, for $k$ dividing $d$ and $b$, the morphism defined by $[(x_1,x_2)]\mapsto [(x_1^k,x_2)]$ induces an isomorphism $X(d;a,b) \simeq X(\frac{d}{k};a,\frac{b}{k})$, and similarly for some $k$ dividing $d$ and $a$. Hence, $X(d;a,b)$ can be normalized with the isomorphism \[X(d;a,b)  \longrightarrow X \bigg( \frac{d}{(d,a)(d,b)}; \frac{a}{(d,a)}, \frac{b}{(d,b)}\bigg): [(x_1,x_2)] \mapsto \big[ (x_1^{(d,b)},x_2^{(d,a)}) \big].\] For general $n \geq 1$, we call a (not necessarily cyclic) type $(\d;A)$ \emph{normalized} if $\mu_{\d}$ is a small subgroup of $GL(n,\C)$ (i.e., it does not contain rotations around hyperplanes other than the identity) acting freely on $(\C^{\ast})^n$ or, equivalently, if for all $\x \in \C^n$ with exactly $n-1$ coordinates different from $0$, the stabilizer subgroup is trivial. It is possible to convert any type into a normalized form. \\

For $n = 2$, we can simplify a normalized type $(d;a,b)$ even further. More precisely, as $\gcd(d,a) = 1$, there exists an integer $a' \in \Z$ with $\gcd(d,a') = 1$ such that $aa' \equiv 1 \mod d$. Then, the space $X(d;a,b)$ is isomorphic to $X(d;a'a,a'b) = X(d;1,a'b)$ under the identity morphism. In other words, every two-dimensional quotient space singularity $(X(\d;A),[0])$ is a \emph{Hirzebruch-Jung singularity} $(\C^2/\mu_d,0)$ where the action of $\mu_d$ on $\C^2$ is given by $(\xi,(x_1,x_2)) \mapsto (\xi x_1,\xi^qx_2)$ for some integer $q \in \{1,\ldots, d-1\}$ with $\gcd(d,q) = 1$. This is called a \emph{Hirzebruch-Jung singularity of type} $\frac{1}{d}(1,q)$. Similarly, we could start with an integer $b' \in \Z$ such that $\gcd(d,b') = 1$ and $bb' \equiv 1 \mod d$. In this case, we find that $X(d;a,b)$ is isomorphic to $X(d;q',1)$, where $q'q \equiv 1 \mod d$. In other words, a Hirzebruch-Jung singularity of some type $\frac{1}{d}(1,q)$ is always equal to the Hirzebruch-Jung singularity of type $\frac{1}{d}(q',1)$ for $q' \in \{1,\ldots, d-1\}$ the unique solution of $qq' \equiv 1 \mod d$. It is well known that the minimal good resolution of a Hirzebruch-Jung singularity has only rational exceptional curves and a bamboo-shaped (i.e., linear) dual graph 
\begin{figure}[ht]
\includegraphics{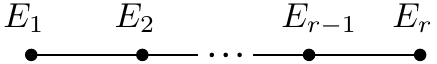}
\end{figure} 

\noindent Furthermore, the self-intersection number $-\kappa_i$ of $E_i$ for $i = 1,\ldots r$ with $\kappa_i \in \N_{\geq 2}$ can be computed from the continued fraction expansion \[\frac{d}{q} = \kappa_r - \frac{1}{\kappa_{r-1} - \frac{1}{\kappa_{r-2} - \cdots}},\] and the positive integers $d$, $q$ and $q'$ are the absolute value of the determinant of the intersection matrix of all exceptional curves, of $E_1,\ldots, E_{r-1}$, and of $E_2,\ldots, E_r$, respectively. \\

Before we can give the precise definition of a good $\Q$-resolution, we still need to introduce two notions: $V$-manifolds and $\Q$-normal crossing divisors. In~\cite{Sat}, a \emph{$V$-manifold} of dimension $n$ was introduced as a complex analytic space admitting an open covering $\{U_i\}$ in which each $U_i$ is analytically isomorphic to some quotient $B_i/G_i$ for $B_i \subseteq \C^n$ an open ball and $G_i$ a finite subgroup of $GL(n,\C)$. We consider $V$-manifolds in which every $G_i$ is a finite abelian subgroup of $GL(n,\C)$, which are normal varieties that can locally be written like $X(\d;A)$. An important example of a $V$-manifold is the \emph{weighted projective space $\mathbb{P}^n_\omega$ of type $\omega$} for some weight vector $\omega = (p_0,\ldots, p_n)$ of positive integers which is defined as the quotient of $\C^{n+1}\setminus\{0\}$ under the action $\C^{\ast} \times (\C^{n+1}\setminus \{0\}) \rightarrow \C^{n+1}\setminus \{0\}$ given by  $(t,(x_0,\ldots,x_n)) \mapsto (t^{p_0}x_0,\ldots, t^{p_n}x_n).$ A two-dimensional $V$-manifold with abelian quotient singularities is also called a \emph{$V$-surface}. A \emph{$\Q$-normal crossing divisor} on a $V$-manifold $X$ is a hypersurface $D$ that is locally isomorphic to the quotient of a normal crossing divisor under an action $(\d;A)$. More precisely, for every point $p \in X$, there exists an isomorphism of germs $(X,p) \simeq (X(\d;A),[0])$ such that $(D,p) \subseteq (X,p)$ is identified with a germ of the form \[(\{[\x] \in X(\d;A)\mid x_1^{m_1}\cdots x_k^{m_k} = 0\},[0]).\] This notion was introduced in~\cite{Ste}. 

\begin{remark}
	In modern language, one usually calls a $V$-manifold an \emph{orbifold}. We keep saying $V$-manifold in this article to emphasize that we follow Steenbrink's approach.
\end{remark}

We can now define a \emph{good $\Q$-resolution} for a germ $(X,0)$ of an isolated singularity as a proper birational morphism $\pi: \tilde{X} \rightarrow X$ such that the following properties hold:
\begin{enumerate}
	\item[(i)] $\tilde{X}$ is a $V$-manifold with abelian quotient singularities;
	\item[(ii)] $\pi$ is an isomorphism over $X \setminus \{0\}$; and
	\item[(iii)] the exceptional divisor $\pi^{-1}(0)$ is a $\Q$-normal crossing divisor on $\tilde{X}$.
\end{enumerate}
For $(Y,0) \subset (X,0)$ a subvariety of codimension one, an \emph{embedded $\Q$-resolution} is a proper birational morphism $\pi: \tilde X \rightarrow X$ with the above three properties in which $X \setminus \{0\}$ is replaced by $X \setminus \text{Sing}(Y)$, and $\pi^{-1}(0)$ by $\pi^{-1}(Y)$. As for a classical good or embedded resolution, we can use the construction of blowing up to compute a good or embedded $\Q$-resolution, but in this case, we use \emph{weighted blow-ups}. Although weighted blow-ups can be placed in the realm of toric resolutions, we follow the approach in~\cite{AMO1} and~\cite{AMO2}. \\

We end this section by briefly discussing an intersection theory on surfaces with abelian quotient singularities. On normal surfaces, an intersection theory was first defined by Mumford~\cite{Mum} and further developed by Sakai~\cite{Sak}; a general intersection theory can be found in~\cite{Ful}. For $V$-manifolds of dimension $2$, which are normal surfaces, an equivalent definition was given in~\cite{AMO2}. Here, we focus on explaining the definitions and properties presented in the latter article that are needed in the present article. First of all, on a $V$-surface $S$, the notions of \emph{Weil and Cartier divisor} coincide after tensoring with $\Q$. More precisely, for every Weil divisor $D$ on $S$, there exists an integer $k \in \Z$ such that $kD$ is locally principal. Therefore, we call the class of divisors on $S$ with rational coefficients modulo linear equivalence the \emph{$\Q$-divisors} on $S$, and we can develop a \emph{rational} intersection theory. In this article, we will only need to compute the \emph{local intersection number} $(D_1\cdot D_2)_p$ of two $\Q$-divisors $D_1$ and $D_2$ at a point $p \in S$. For this purpose, we assume that $p$ is the origin $[0]$ in a normalized cyclic quotient space $X(d;a,b)$, that $D_i = \{f_i = 0\}$ for $i = 1,2$ is given by a reduced polynomial in $\C[x,y]$, that the support of $D_1$ is not contained in the support of $D_2$, and that $D_1$ is irreducible. In this case, the local intersection number at $p$ is well-defined and given by 
\begin{equation} \label{eq:local-intersection-number}
	(D_1\cdot D_2)_p := \frac{1}{d}\dim_{\C}\left(\frac{\C\{x,y\}}{\langle f_1,f_2\rangle}\right) \in \Q.
\end{equation}
Another property of the intersection product that we will use is that for $\pi: \tilde X \rightarrow X(d;a,b)$ a weighted blow-up at the origin with exceptional divisor $E$, and for $D$ a $\Q$-divisor on $X(d;a,b)$, we have 
\begin{equation}\label{eq:intersection-pull-back}
	\pi^*D \cdot E = 0.
\end{equation}
This can be shown in the same way as the analogous statement for the classical blow-up. 

\section{Our family of normal surface singularities}\label{sec:surfaces}

In this section, we introduce the family of normal surface singularities of our interest that appear in the proof from \cite{MVV1} of the monodromy conjecture for a space monomial curve introduced in Section~\ref{sec:space-monomial}. We also introduce a good $\Q$-resolution, which we immediately use to show that these singularities for $g\geq 3$ are not Brieskorn-Pham and to show the conditions for their link to be a $\q$. In the next section, we will use the same resolution to identify the singularities in this family with a $\z$ link, and to show that these are of splice type. 

\subsection{Definition of our surface singularities}\label{sec:def-surfaces}

Consider a space monomial curve $Y \subset \C^{g+1}$ given by the equations~\eqref{eq:equations-Y} with $g \geq 2$. For $(\lambda_2,\ldots, \lambda_g) \in (\C\setminus \{0\})^{g-1}$, we define the affine scheme $S(\lambda_2,\ldots, \lambda_g)$ in $\C^{g+1}$ given by 
\begin{equation} \label{eq:equations-S}
		\renewcommand{\arraystretch}{1.1}{\left\{\begin{array}{c c l l l}
		f_1& + & \lambda_2f_2  &  = 0 \\
		f_2& + & \lambda_3f_3 &= 0  \\
		& \vdots & &\\
		f_{g-1} & +  & \lambda_gf_g & = 0.\\
    	\end{array}\right.}
\end{equation}
For \emph{generic} $(\lambda_2,\ldots, \lambda_g) \in (\C\setminus \{0\})^{g-1}$ (i.e., the point $(\lambda_2,\ldots, \lambda_g)$ is contained in the non-zero complement of a specific closed subset of $(\C\setminus \{0\})^{g-1}$), one can show that $S(\lambda_2,\ldots, \lambda_g)$ is a normal complete intersection surface which is smooth outside the origin, see~\cite[Prop. 4.2]{MVV1}. From now on, we will denote such a surface by $S := S(\lambda_2,\ldots, \lambda_g) \subset \C^{g+1}$, and we are interested in the link of these normal singularities $(S,0) \subset (\C^{g+1},0)$. 

\subsection{A good $\Q$-resolution of our surface singularities} 	\label{sec:Q-resolution} 

In \cite[Section 5]{MVV1}, the computation of $g$ weighted blow-ups $\varphi_k$ for $k = 1,\ldots, g$ yields an embedded $\Q$-resolution $\varphi = \varphi_1 \circ \cdots \circ \varphi_g : \hat{S} \rightarrow S$ of the space monomial curve $Y$ given by~\eqref{eq:equations-Y} seen as Cartier divisor on $S$. Because the surface $S$ is already $\Q$-resolved after the first $g-1$ blow-ups, and the last step is needed to desingularize the curve $Y$, we can consider the good $\Q$-resolution $\hat{\varphi} := \varphi_1 \circ \cdots \circ \varphi_{g-1}: \hat{S} \rightarrow S$ of $(S,0)$. We will now explain the properties of this resolution that are needed to see that $(S,0)$ is not Brieskorn-Pham for $g \geq 3$, and to prove the characterization for $(S,0)$ to have a $\q$ or $\z$ link from Theorem~\ref{thm:rat-int-hom}. For more details, we refer to~\cite[Section 5]{MVV1}. \\

First of all, for each blow-up $\varphi_k$ for $k = 1,\ldots, g-1$, we denote the exceptional divisor by $\E_k$. To ease the notation, we also denote their strict transform under later blow-ups by $\E_k$. Hence, in the end, the exceptional curves of the good $\Q$-resolution $\hat{\varphi}$ are the irreducible components of these $\E_k$. If we define \[r_k := \frac{e_k}{\lcm(n_{k+1},\ldots, n_g)}, \qquad k = 1,\ldots, g-1,\] then each $\E_k$ is the disjoint union of $r_k$ isomorphic irreducible components that we denote by $\E_{kj}$ for $j = 1,\ldots, r_k$. In particular, the last exceptional divisor $\E_{g-1}$ is always irreducible, and the pull-back of the Cartier divisor $Y$ under $\hat{\varphi}$ is given by 
\begin{equation}\label{eq:pull-back-Y}
 	\hat{\varphi}^{\ast}Y = \hat{Y} + \sum_{\smallmatrix 1 \leq k \leq g-1 \\ 1 \leq j \leq r_k \endsmallmatrix} N_k \E_{kj},
 \end{equation}
where $\hat{Y}$ is the strict transform of $Y$ under $\hat{\varphi}$, and $N_k$ for $k = 1,\ldots, g-1$ is the multiplicity of $\E_k$, which is equal to $\lcm ( \frac{\lbeta_k}{e_k}, n_k, \ldots, n_g )$. Furthermore, each divisor $\E_k$ for ${k = 2,\ldots, g-2}$ (if $g\geq 4$) only intersects $\E_{k-1}$ and $\E_{k+1}$, and $\E_{g-1}$ only intersects $\E_{g-2}$. For every $k = 1,\ldots, g-2$ (if $g \geq 3$), the intersections of $\E_k$ and $\E_{k+1}$ are \emph{equally distributed}, that is, each of the components $\E_{(k+1)j}$ of $\E_{k+1}$ intersects precisely $\frac{r_k}{r_{k+1}}$ components of $\E_k$, each component $\E_{kj}$ of $\E_k$ is intersected by only one of the components of $\E_{k+1}$, and each non-empty intersection between two components $\E_{kj}$ and $\E_{(k+1)j'}$ consists of a single point. In other words, the dual graph of the good $\Q$-resolution $\hat{\varphi}: \hat{S} \rightarrow S$ is a tree as in Figure~\ref{fig:dual-graph}. 
  
\begin{figure}[ht]
\includegraphics{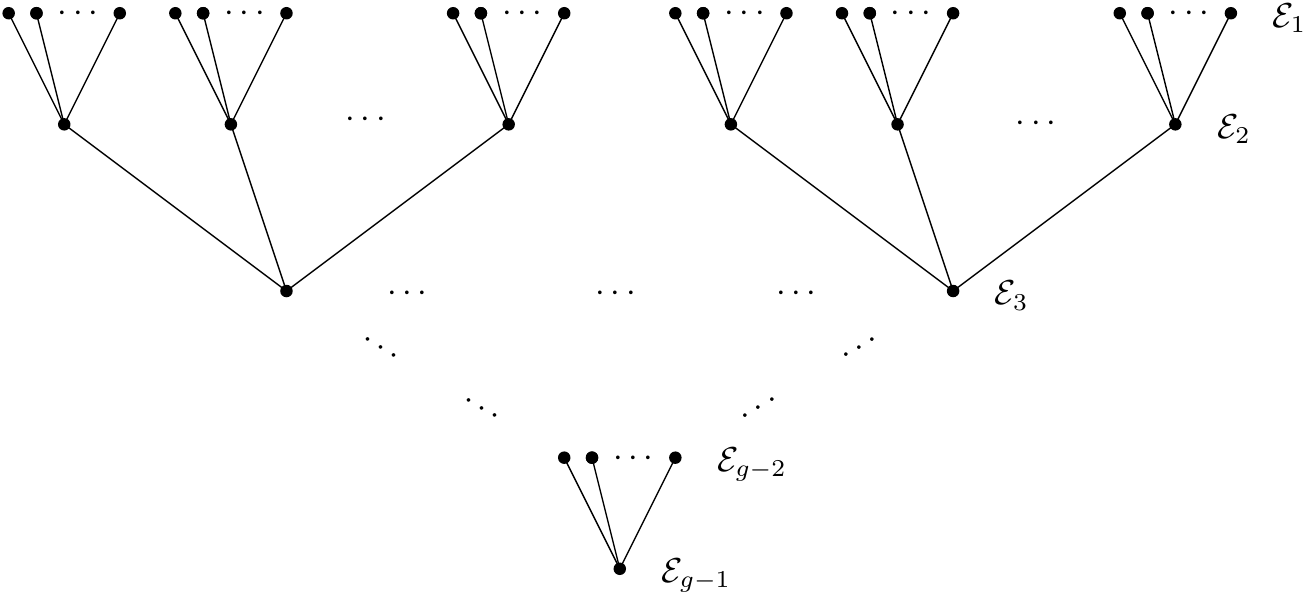}
\caption{Dual graph of the good $\Q$-resolution of $(S,0)$.}
\label{fig:dual-graph}
\end{figure}

It is important to note that $\hat{\varphi}$ is not a good resolution of $(S,0)$ as $\hat{S}$ still contains a lot of singularities that need to be resolved. To explain these singularities, we put $M_k := \lcm(\frac{\lbeta_k}{e_k}, n_{k+1},\ldots, n_g)$ for $k = 0,\ldots, g$, and we consider the divisors $H_i$ for $i = 0,\ldots, g$ on $S$ defined by $\{ x_i = 0 \} \cap S \subset \C^{g+1}$. Again, to ease the notation, we denote their strict transforms also by $H_i$ throughout the process. We further consider the curve $Y$ whose strict transform is always denoted by $\hat{Y}$. In the resolution of $(S,0)$, each $H_k$ for $k = 1,\ldots, g-1$ is separated from $\hat{Y}$ at the $k$th step and intersects the $k$th exceptional divisor $\E_k$ transversely at some singular point(s). More precisely, if we denote a point in the intersection $\E_k \cap H_k$ by $Q_k$, then there are $\frac{\lbeta_k}{M_k}$ such points which are equally distributed along the $r_k$ components of $\E_k$. Locally around each such point, we have the following situation at $[(x_0,x_k)]$:
\begin{equation*}
	\left\{\begin{aligned}
		& \hat{S} = X \bigg( \gcd \Big( e_{k-1}, \frac{n_k \bar{\beta}_k}{n_{k+1}}, \ldots,\frac{n_k \bar{\beta}_k}{n_g} \Big);-1,\lbeta_k\bigg) \\
		& \E_k: \ x_0^{n_k \bar{\beta}_k} = 0, \qquad H_k: \ x_k = 0.
	\end{aligned}\right.
\end{equation*}\[\]
Because $\gcd\big(\gcd ( e_{k-1}, \frac{n_k \bar{\beta}_k}{n_{k+1}}, \ldots,\frac{n_k \bar{\beta}_k}{n_g}),\lbeta_k\big) = \gcd \big( e_k, \frac{n_k \bar{\beta}_k}{n_{k+1}}, \ldots,\frac{n_k \bar{\beta}_k}{n_g} \big),$ these points are Hirzebruch-Jung singularities of type $\frac{1}{d_k}(1,q_k)$ with
\[d_k := \frac{\gcd \Big( e_{k-1}, \frac{n_k \bar{\beta}_k}{n_{k+1}}, \ldots,\frac{n_k \bar{\beta}_k}{n_g} \Big)}{\gcd\Big( e_k, \frac{n_k \bar{\beta}_k}{n_{k+1}}, \ldots,\frac{n_k \bar{\beta}_k}{n_g} \Big)} = \frac{\lcm\Big(\frac{n_k\lbeta_k}{e_k}, n_{k+1},\ldots, n_g\Big)}{\lcm\Big(\frac{n_k\lbeta_k}{e_{k-1}}, n_{k+1},\ldots, n_g\Big)} = \frac{N_k}{M_k}.\] Here, the second equality follows from the elementary fact that for $m_1,\ldots, m_r$ a set of non-zero integers and $m$ a common multiple, we have 
\begin{equation} \label{eq:rel-gcd-lcm}
	\gcd\Big(\frac{m}{m_1}, \ldots, \frac{m}{m_r} \Big) = \frac{m}{\lcm(m_1,\ldots, m_r)},
\end{equation}
and the third equality follows from the definition $n_k = \frac{e_{k-1}}{e_k}$ and the fact that $\gcd(\frac{\lbeta_k}{e_k},n_k) = 1$. For later purposes, we can rewrite $d_k$ as 
\begin{equation}\label{eq:d_k}
 	d_k = \frac{n_k\gcd\Big(\frac{\lbeta_k}{e_k},\lcm(n_{k+1},\ldots,n_g)\Big)}{\gcd\Big(\frac{n_k\lbeta_k}{e_k},\lcm(n_{k+1},\ldots, n_g)\Big)} = \frac{n_k}{\gcd\big(n_k,\lcm(n_{k+1},\ldots, n_g)\big)} = \frac{n_kr_k}{r_{k-1}},
\end{equation}
where we extend the sequence $r_1,\ldots, r_{g-1}$ with $r_0 := \frac{e_0}{\lcm(n_1,\ldots, n_g)} = \frac{\lbeta_0}{M_0}$. \\

Similarly, in the intersection $\E_1 \cap H_0$, there are $\frac{\lbeta_0}{M_0}$ points denoted by $Q_0$ around which we have the following local equation at $[(x_0,x_1)]$:
\begin{equation*}
	\left\{\begin{aligned}
		& \hat{S} = X \bigg( \gcd\Big(\frac{n_0\lbeta_0}{n_1},\frac{n_0\lbeta_0}{n_2},\ldots,\frac{n_0\lbeta_0}{n_g}\Big);\lbeta_0,-1 \bigg) \\
		& \E_1: \ x_1^{n_0\lbeta_0} = 0, \qquad H_0: \ x_0 = 0.
	\end{aligned}\right.
\end{equation*}
Each component of $\E_1$ contains the same number of such points, which are of type $\frac{1}{d_0}(q_0,1)$ with \[d_0 := \frac{\gcd\Big(\frac{n_0\lbeta_0}{n_1},\frac{n_0\lbeta_0}{n_2},\ldots,\frac{n_0\lbeta_0}{n_g}\Big)}{\gcd\Big(\lbeta_0,\frac{n_0\lbeta_0}{n_1},\ldots,\frac{n_0\lbeta_0}{n_g}\Big)} = \frac{N_1}{M_0},\] where we again used relation~\eqref{eq:rel-gcd-lcm} and the fact that $\frac{\lbeta_1}{e_1} = n_0$. \\

For $g \geq 3$, a next set of singular points of $\hat{S}$ are the points in an intersection $\E_k \cap \E_{k+1}$ for ${k = 1,\ldots, g-2}$ that we denote by $Q_{k(k+1)}$. We have already explained that there are $r_k$ such points in total, one on each component of $\E_k$, and that each component of $\E_{k+1}$ contains $\frac{r_k}{r_{k+1}}$ such points. Furthermore, the local situation around $Q_{k(k+1)}$ can be described in the variables $[(x_0,x_{k+1})]$ by:
\begin{equation}\label{eq:local-situation-Q_{k(k+1)}}
	\left\{\begin{aligned}
		& \hat{S} = X \left(\!\!\! \begin{array}{c|cc}
			\frac{n_{k+1} \bar{\beta}_{k+1} - n_k \bar{\beta}_k}{\lcm(n_{k+1},\ldots,n_g)} & 1 & -1 \\[0.2cm]
			(n_{k+1} \bar{\beta}_{k+1} - n_k \bar{\beta}_k) e_{k+1} & - \bar{\beta}_{k+1} & \frac{n_k\lbeta_k}{n_{k+1}}
			\end{array} \!\!\right) \\[5pt]
		& \E_k: \ x_0^{n_k \bar{\beta}_k} = 0, \qquad \E_{k+1}: \ x_{k+1}^{n_{k+1} \bar{\beta}_{k+1}} = 0.
	\end{aligned}\right.
\end{equation}
One can show that these are cyclic quotient singularities with 
\begin{equation}\label{eq:d_{k(k+1)}}
	d_{k(k+1)} := \frac{r_kN_kN_{k+1}(n_{k+1}\lbeta_{k+1} - n_k\lbeta_k)}{n_kn_{k+1}\lbeta_k\lbeta_{k+1}}
\end{equation}
the order of the underlying small group as follows. First, by multiplying conveniently, we can rewrite $\hat{S}$ into the form $X(\begin{smallmatrix} d \\ d \end{smallmatrix} \vert \begin{smallmatrix} a_1 & a_2 \\ a_3 & a_4 \end{smallmatrix})$, where $d = n_{k+1}\lbeta_{k+1} -  n_k\lbeta_k$. Second, the group automorphism $(\xi,\eta) \mapsto (\xi\eta^{-1},\eta)$ on $\mu_d \times \mu_d$ induces an isomorphism $X(\begin{smallmatrix} d \\ d \end{smallmatrix} \vert \begin{smallmatrix} a_1 & a_2 \\ a_3 & a_4 \end{smallmatrix})\simeq X(\begin{smallmatrix} d \\ d \end{smallmatrix} \vert \begin{smallmatrix} a_1 & a_2 \\ a_3 - a_1 & a_4 - a_2 \end{smallmatrix})$ given by the identity. Using such an automorphism repeatedly yields an isomorphism 
\begin{equation}\label{eq:upper-triangular}
	X\left(\begin{array}{c|cc} d & a_1 & a_2 \\ d & a_3 & a_4 \end{array}\right)\simeq X\left(\begin{array}{c|cc} d & \gcd(a_1,a_3) & \alpha a_2 + \beta a_4\\ d & 0 & \frac{a_1a_4 - a_2a_3}{\gcd(a_1,a_3)} \end{array}\right): [(x_1,x_2)] \mapsto [(x_1,x_2)],
\end{equation}
where $\alpha, \beta\in \Z$ such that $\gcd(a_1,a_3) = \alpha a_1 + \beta a_3$. Third, every quotient space of the form $X(\begin{smallmatrix} d \\ d \end{smallmatrix} \vert \begin{smallmatrix} a_1 & a_2 \\ 0 & a_4 \end{smallmatrix}) $ is isomorphic to a cyclic quotient space under the morphism \[X\left(\begin{array}{c|cc} d & a_1 & a_2 \\ d & 0 & a_4 \end{array}\right)\simeq X\Big(d; a_1, \frac{da_2}{\gcd(d,a_4)}\Big): [(x_1,x_2)] \mapsto [(x_1,x_2^{\frac{d}{\gcd(d,a_4)}})].\] Finally, we can rewrite the resulting cyclic singularity into a Hirzebruch-Jung singularity as explained in Section~\ref{sec:quotient-sing}. We do not provide more details as we will not need an explicit expression for $d_{k(k+1)}$ in general. It is, however, worth mentioning that this approach can be used to show that any quotient space $X({\bf d};A) = \C^2/\mu_{\bf d}$ is isomorphic to a cyclic quotient space, and that we will illustrate this approach when the link of $(S,0)$ is a $\z$, see Section~\ref{sec:our-surfaces-splice-type}. \\
 
The last singular point of $\hat{S}$ for $g\geq 2$ is the intersection point $P_{g-1} := \E_{g-1} \cap \hat{Y} = \E_{g-1} \cap H_g$ around which we have 
\begin{equation}\label{eq:local-situation-P_{g-1}}
	\left\{\begin{aligned}
		& \hat{S} = X\bigg(n_g;-1, \frac{n_{g-1} \bar{\beta}_{g-1}}{n_g}\bigg)\\
		& \E_{g-1}: \ x_0^{n_{g-1} \bar{\beta}_{g-1}} = 0, \qquad H_g : \ x_g = 0, \qquad \hat{Y}: \ x_g^{n_g} - x_0^{n_g\lbeta_g - n_{g-1}\lbeta_{g-1}} = 0,
	\end{aligned}\right.
\end{equation}
Clearly, this point is a Hirzebruch-Jung singularity of type $\frac{1}{d}(1,q)$ with \[d := \frac{n_g}{\gcd(n_{g-1},n_g)\gcd\Big(\frac{\lbeta_{g-1}}{n_g},n_g\Big)}.\]

To recapitulate, we visualize the good $\Q$-resolution $\hat{\varphi}$ as in Figure~\ref{fig:final-resolution}, which shows the exceptional curves and the singular points. For simplicity, the components of each $\E_k$ are represented by lines, but we will see in a moment that they are not rational in general. \\

\begin{figure}[ht]
\includegraphics{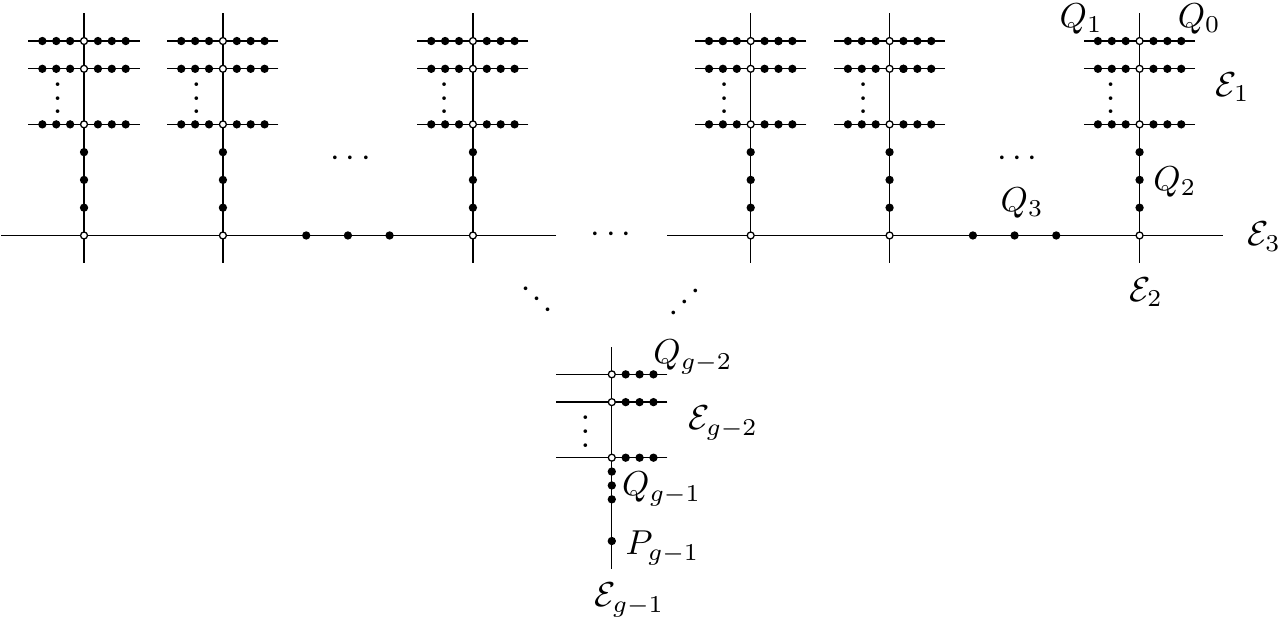}    
\caption{The good $\Q$-resolution of $(S,0)$.}
\label{fig:final-resolution}
\end{figure}

Using Corollary 6.5 from~\cite{MVV1}, we can compute the Euler characteristic of the exceptional curves of $\hat{\varphi}$. More precisely, this result gives an expression for the Euler characteristic of the exceptional divisor $\E_k$ for $k =1,\ldots, g-1$ without its singularities: it states that
\[\check{\E}_k := \left\{\begin{array}{ll}
		\E_1 \setminus ((\E_1 \cap H_0) \cup (\E_1\cap H_1) \cup (\E_1\cap \E_2)) & \text{for } k = 1 \\
		\E_k \setminus ((\E_k \cap H_k) \cup (\E_k\cap \E_{k-1}) \cup (\E_k \cap \E_{k+1})) & \text{for } k = 2,\ldots, g-2 \\
		\E_{g-1} \setminus ((\E_{g-1} \cap H_{g-1}) \cup (\E_{g-1} \cap \E_{g-2}) \cup (\E_{g-1} \cap \hat{Y})) & \text{for } k = g-1 \\
\end{array}\right.\]
has Euler characteristic $\chi(\check{\E}_k) = - \frac{n_k\lbeta_k}{N_k}.$ Hence, the Euler characteristic of $\E_k$ can be easily computed by adding the cardinality of all its singularities. This yields:
\begin{equation*}\renewcommand{\arraystretch}{1.5}
	\chi(\E_k) = \left\{\begin{array}{ll}
		- \frac{n_1\lbeta_1}{N_1} + \frac{\lbeta_0}{M_0} + \frac{\lbeta_1}{M_1} + r_1
		& \text{for } k = 1 \\
		-\frac{n_k\lbeta_k}{N_k} + \frac{\lbeta_k}{M_k} + r_{k-1} + r_k & \text{for } k = 2,\ldots, g-2\\
		- \frac{n_{g-1}\lbeta_{g-1}}{N_{g-1}} + \frac{\lbeta_{g-1}}{M_{g-1}} + r_{g-2} + 1  & \text{for } k = g-1. 
\end{array}\right.
\end{equation*}
Because the components $\E_{kj}$ for $j = 1,\ldots, r_k$ are disjoint and isomorphic, their Euler characteristic is equal to $\chi(\E_{kj}) = \frac{\chi(\E_k)}{r_k}$. Using that $\frac{N_k}{M_k} = \frac{n_kr_k}{r_{k-1}} = \frac{n_k}{\gcd(n_k,\lcm(n_{k+1},\ldots, n_g))}$ for $k = 1, \ldots, g-1$, see~\eqref{eq:d_k}, and that $\frac{\lbeta_k}{r_kM_k} = \gcd\big(\frac{\lbeta_k}{e_k},\lcm(n_{k+1},\ldots, n_g)\big)$, we can rewrite these Euler characteristics as \begin{equation*}
	\chi(\E_{kj}) =
		2 - \bigg(\gcd\big(n_k,\lcm(n_{k+1},\ldots, n_g)\big) - 1\bigg)\bigg(\gcd\Big(\frac{\lbeta_k}{e_k},\lcm(n_{k+1},\ldots, n_g)\Big) - 1\bigg).
\end{equation*}
We indeed see that the exceptional curves are not rational in general. Even more, this implies that the genus of $\E_{kj}$ is zero if and only if \[\gcd\big(n_k,\lcm(n_{k+1},\ldots, n_g)\big) = 1 \text{ or }\gcd\Big(\frac{\lbeta_k}{e_k},\lcm(n_{k+1},\ldots, n_g)\Big) =1 \text{ for } k = 1,\ldots, g-1.\] Since the dual graph of the good $\Q$-resolution $\hat{\varphi}:\hat{S} \rightarrow S$ is a tree and the quotient singularities of $\hat{S}$ can be resolved with bamboo-shaped dual graphs and rational exceptional curves, these are already the conditions under which $(S,0)$ has a $\q$ link. In other words, we have already shown the first part of Theorem~\ref{thm:rat-int-hom}.

\subsection{Our surface singularities versus Brieskorn-Pham surface singularities}

If $g = 2$, then $(S,0) \subset (\C^3,0)$ is a Brieskorn-Pham surface singularity given by the equation \[u(x_0,x_1)x_0^{n_0} + x_1^{n_1} + \lambda_2x_2^{n_2} = 0,\] where $u(x_0,x_1) = -1 - \lambda_2x_0^{b_{20}-n_0}x_1^{b_{21}} \in \C\{x_0,x_1\}$ is a unit as $b_{20} > n_0$, see~\cite[Lemma 3.2]{MVV2}. 
Hence, in this case, the link of $(S,0)$ is a $\q$ (resp. $\z$) under the condition of Proposition~\ref{prop:Brieskorn-Pham-rat-int}, which is equivalent to (recall that $\gcd(n_0,n_1) = 1$) the condition that $\gcd(n_0,n_2) = 1$ or $\gcd(n_1,n_2)=1$ (resp. that the exponents $n_i$ for $i  = 0,1,2$ are pairwise coprime).\\

If $g\geq 3$, then we claim that $(S,0) \subset (\C^{g+1},0)$ is never a Brieskorn-Pham singularity. To prove this, we will show that the minimal good resolution of $(S,0)$ contains at least $g-1$ \emph{rupture} exceptional curves. An irreducible exceptional curve is called \emph{rupture} if either its genus is positive, or its genus is zero and it has valency at least $3$ (i.e., it intersects at least three times other components of the exceptional locus). This implies that $(S,0)$ is indeed not Brieskorn-Pham for $g \geq 3$ as a Brieskorn-Pham surface singularity has at most one rupture exceptional curve in its minimal good resolution. The latter can be seen by considering a good $\Q$-resolution of a Brieskorn-Pham surface singularity consisting of one weighted blow-up at the origin which yields one irreducible exceptional curve $\E$ containing three sets of Hirzebruch-Jung singularities. We refer for more details to~\cite[Example 3.6]{Ma1}; see also Remark~\ref{rem:Brieskorn-Pham}. As each of these singularities can be minimally resolved with a bamboo-shaped dual graph and rational exceptional curves, the only possible rupture exceptional curve in the obtained good resolution is the strict transform of $\E$. This implies that the minimal good resolution of a Brieskorn-Pham singularity indeed contains at most one rupture exceptional curve. Even more, the minimal good resolution of a Brieskorn-Pham surface singularity has no rupture exceptional curve if and only if it has only rational exceptional curves and a bamboo-shaped dual graph or, thus, if and only if the singularity is a cyclic quotient singularity. \\

To show that the minimal good resolution of $(S,0)$ has at least $g-1$ rupture exceptional curves, we make use of the good $\Q$-resolution $\hat{\varphi}: \hat{S} \rightarrow S$ of $(S,0)$, from which we can obtain a (not necessarily minimal) good resolution $\pi: \tilde S \rightarrow S$ of $(S,0)$ by minimally resolving the singularities of $\hat{S}$. Since these singularities are all Hirzebruch-Jung, the only possible rupture exceptional curves of $\pi$ are the strict transforms of the exceptional curves of the good $\Q$-resolution. The next result immediately implies that the good resolution $\pi$ has at least $g-1$ exceptional divisors that are rupture.

\begin{prop}\label{prop:rupture-divisors}
Let $(S,0) \subset (\C^{g+1},0)$ be a normal surface singularity defined by the equations~\eqref{eq:equations-S} with $g \geq 3$. Consider the good $\Q$-resolution $\hat{\varphi}:\hat{S} \rightarrow S$ of $(S,0)$ introduced in Section~\ref{sec:Q-resolution}. Then,
\begin{enumerate}[wide,labelindent = 0pt]
	\item[(i)] each exceptional curve $\E_{kj}$ for $k = 1,\ldots, g-2$ and $j = 1,\ldots, r_k$ yields a rupture exceptional curve in the good resolution $\pi:\tilde S \rightarrow S$ of $(S,0)$ coming from $\hat{\varphi}$; and
	\item[(ii)] if $r_{g-2} = 1$ (i.e., the exceptional divisor $\E_{g-2}$ is irreducible), then $\E_{g-1}$ yields a rupture exceptional curve in the good resolution $\pi:\tilde S \rightarrow S$ of $(S,0)$ coming from $\hat{\varphi}$.
\end{enumerate} \end{prop}

\begin{proof}
Note that we can determine whether the strict transform of an exceptional curve of $\hat{\varphi}$ is rupture on $\tilde S$ by considering the original exceptional curve on $\hat{S}$ and counting each singularity as an intersection. However, we need to take into account that, under certain conditions, it is possible that some of the quotient singularities are in fact smooth. In this case, the latter points can not be counted as an intersection in the good resolution. \\

Let us first consider a component $\E_{1j}$ for some $j \in \{1,\ldots ,r_1\}$. If its genus is positive, then it will trivially induce a rupture exceptional curve. So suppose that its genus is zero, that is, $\gcd(n_1,\lcm(n_2,\ldots, n_g)) = 1$ or $\gcd(\frac{\lbeta_1}{e_1},\lcm(n_2,\ldots, n_g)) =1$. Since $\E_{1j}$ intersects $\E_2$ in a single point, we need to show that it contains at least two actual singular points outside $\E_2$. Recall that $\E_{1j}$ contains $\frac{\lbeta_0}{r_1M_0} = \gcd(n_1,\lcm(n_2,\ldots, n_g))$ points $Q_0$ whose order as Hirzebruch-Jung singularity is $d_0 = \frac{N_1}{M_0}$, and $\frac{\lbeta_1}{r_1M_1} = \gcd(\frac{\lbeta_1}{e_1},\lcm(n_2,\ldots, n_g))$ points $Q_1$ with order $d_1 = \frac{N_1}{M_1} = \frac{n_1}{\gcd(n_1,\lcm(n_2,\ldots n_g))}$. If $\gcd(n_1,\lcm(n_2,\ldots, n_g)) = 1$, then $d_1 = n_1$; if $\gcd(\frac{\lbeta_1}{e_1},\lcm(n_2,\ldots, n_g)) = 1$, then $d_0 = \frac{\lbeta_1}{e_1} = n_0$. Hence, we can distinguish three cases:
\begin{enumerate}[wide,labelindent = 0pt]
	\item[(i)] if $\gcd(n_1,\lcm(n_2,\ldots, n_g)) = 1$ and $\gcd(\frac{\lbeta_1}{e_1},\lcm(n_2,\ldots, n_g)) \geq 2$, then $\E_{1j}$ contains at least two singular points $Q_1$ with order $d_1 = n_1 > 1$;
	\item[(ii)] if  $\gcd(n_1,\lcm(n_2,\ldots, n_g))\geq 2$ and $\gcd(\frac{\lbeta_1}{e_1},\lcm(n_2,\ldots, n_g)) = 1$, then $\E_{1j}$ contains at least two singular points $Q_0$ with order $d_0 = n_0 > 1$;
	\item[(iii)] if $\gcd(n_1,\lcm(n_2,\ldots, n_g)) = \gcd(\frac{\lbeta_1}{e_1},\lcm(n_2,\ldots, n_g)) =1$, then $\E_{1j}$ contains one singular point $Q_0$ with order $d_0 = n_0 > 1$ and one singular point $Q_1$ with order $d_1 = n_1 > 1.$
\end{enumerate}
In other words, $\E_{1j}$ will indeed always yield a rupture exceptional curve. \\

For $\E_{kj}$ with $k \in \{2,\ldots, g-2\}$ (if $g \geq 4$) and $j \in \{1,\ldots, r_k\}$, we can work in a similar way. Assume again that its genus is zero, which is now the case if and only if $\gcd(n_k,\lcm(n_{k+1},\ldots, n_g)) = 1$ or $ \gcd(\frac{\lbeta_k}{e_k},\lcm(n_{k+1},\ldots, n_g)) = 1$. We know that $\E_{kj}$ has $\frac{r_{k-1}}{r_k} = \gcd(n_k,\lcm(n_{k+1},\ldots, n_g))$ intersection points with $\E_{k-1}$, a single intersection point with $\E_{k+1}$ and $\frac{\lbeta_k}{r_kM_k} = \gcd(\frac{\lbeta_k}{e_k},\lcm(n_{k+1},\ldots, n_g))$ points $Q_k$ whose order as Hirzebruch-Jung singularity is $d_k = \frac{N_k}{M_k} = \frac{n_k}{\gcd(n_k,\lcm(n_{k+1},\ldots, n_g))}$. Hence, if $\gcd(n_k,\lcm(n_{k+1},\ldots, n_g)) \geq 2$ (and $\gcd(\frac{\lbeta_k}{e_k},\lcm(n_{k+1},\ldots, n_g)) = 1$), then $\E_k$ has at least three intersections with other exceptional curves of $\hat{\varphi}$, and we are done. If $\gcd(n_k,\lcm(n_{k+1},\ldots, n_g)) = 1$, then $d_k = n_k > 1$. Therefore, in this case, $\E_{kj}$ will also be rupture as it intersects both $\E_{k-1}$ and $\E_{k+1}$ in a single point and contains at least one singular point $Q_k$ with order $d_k > 1$. \\

It remains to show the second part. If $r_{g-2} = 1$, then $\gcd(n_{g-1},n_g) = 1$, which implies that $\E_{g-1}$ has zero genus. Furthermore, it has one intersection point with $\E_{g-2}$, one point $P_{g-1}$ with order \[d = \frac{n_g}{\gcd\Big(\frac{\lbeta_{g-1}}{n_g},n_g\Big)},\] and $\frac{\lbeta_{g-1}}{M_{g-1}} = \gcd(\frac{\lbeta_{g-1}}{n_g},n_g)$ points $Q_{g-1}$ with order $d_{g-1} = n_{g-1} > 1$. We can again conclude: if $\gcd(\frac{\lbeta_{g-1}}{n_g},n_g) \geq 2$, then $\E_{g-1}$ contains at least two singular points $Q_{g-1}$ with order $d_{g-1} > 1$; if $\gcd(\frac{\lbeta_{g-1}}{n_g},n_g) = 1$, then $\E_{g-1}$ contains exactly two singular points, namely one $Q_{g-1}$ with order $d_{g-1} > 1$, and $P_{g-1}$ with order $d = n_g > 1$.  
\end{proof}

We still need to show that the minimal good resolution of $(S,0)$ contains at least $g-1$ rupture exceptional curves. From Proposition~\ref{prop:rupture-divisors}, it follows that each exceptional curve $\E_{kj}$ for $k = 1,\ldots, g-2$ and $j = 1,\ldots, r_k$ can not be contracted in the good resolution $\pi:\tilde S \rightarrow S$; either its genus is positive so that Castelnuovo's Contractibility Theorem does not apply, or it has at least three intersections with other exceptional curves so that the exceptional locus would not be a simple normal crossing divisor after contracting $\E_{kj}$. The same applies to $\E_{g-1}$ if $r_{g-2} = 1$ or $r_{g-2} \geq 3$. In other words, in these cases, the good resolution $\pi$ is minimal. If $r_{g-2} = 2$, it is possible that $\E_{g-1}$ is superfluous as the next example shows. However, the obtained minimal good resolution of $(S,0)$ coming from contracting $\E_{g-1}$ (and possibly executing subsequent contractions) will still have at least $g-1$ rupture exceptional curves: all the exceptional curves $\E_{kj}$ for $k=1,\ldots, g-2$ and $j = 1,\ldots, r_k$  are rupture, where $r_k \geq 1$ for $k = 1,\ldots, g-3$ (if $g \geq 4$) and $r_{g-2} = 2$.

\begin{ex}\label{ex:E_{g-1}-superfluous}
If $r_{g-2} = 2$, then it is possible that the good resolution $\pi:\tilde S \rightarrow S$ is not minimal. For example, consider the surface $S \subset \C^4$ defined by 
\begin{equation}\label{eq:example}
\renewcommand{\arraystretch}{1.2}{\left\{\begin{array}{l c l l l}
	x_1^2-x_0^3 & + & x_2^2 - x_0^5x_1 &  = 0 \\
	x_2^2 - x_0^5x_1 & + & x_3^2 - x_0^{10}x_2 &= 0.
    \end{array}\right.}	
\end{equation}
The semigroup of the corresponding space monomial curve $Y \subset \C^4$ is minimally generated by $(8,12,26,53)$. From the properties of the good $\Q$-resolution $\hat{\varphi}$ explained above, one can easily check the following:
\begin{enumerate}
	\item[(i)] the first exceptional divisor $\E_1$ has $r_1 = 2$ components $\E_{11}$ and $\E_{12}$ that each contain two singular points $Q_0$ of type $\frac{1}{3}(1,1)$, while every point $Q_1$ is smooth; 
	\item[(ii)] the genus of $\E_2$ is zero, and the points $P_2$ and $Q_2$ are smooth; and
	\item[(iii)] the intersection of $\E_1$ and $\E_2$ consists of two singular points $Q_{12}$, one on each component of $\E_1$, that are Hirzebruch-Jung of type $\frac{1}{7}(1,3)$.
\end{enumerate}
It follows that the dual graph of $\pi:\tilde S \rightarrow S$ is as in Figure~\ref{fig:dual-graph-example}, where we denote the strict transforms of $\E_{1j}$ and $\E_2$ still by $\E_{1j}$ and $\E_2$, respectively, and where the exceptional curves $\E_j^0$ and $\E^{12}_j$ come from resolving the singularities $Q_0$ and $Q_{12}$, respectively. Furthermore, one can show that the pull-back of $Y$ is given \[\pi^{\ast}Y = \hat{Y} + 6\sum_{j = 1}^2\E_{1j} + 26 \E_2 + 2\sum_{j=1}^4\E_j^0 + 8\sum_{j=1}^2\E^{12}_j + 10\sum_{j=3}^4\E^{12}_j + 12\sum_{j=5}^6\E^{12}_j,\] where $\hat{Y}$ is the strict transform of $Y$. Because $\pi^{\ast}Y \cdot \E_2 = 0$ by~\eqref{eq:intersection-pull-back} and $\hat{Y} \cdot \E_2 = 2$, which can be seen from the local equation~\eqref{eq:local-situation-P_{g-1}}, we find that the self-intersection number of $\E_2$ is $-1$. Hence, by Castelnuovo's Contractibility Theorem, the exceptional curve $\E_2$ can be contracted in order to find the minimal good resolution of $(S,0)$. However, this minimal good resolution has still $g-1 = 2$ rupture exceptional curves, namely $\E_{11}$ and $\E_{12}$.

\begin{figure}[ht]
\includegraphics{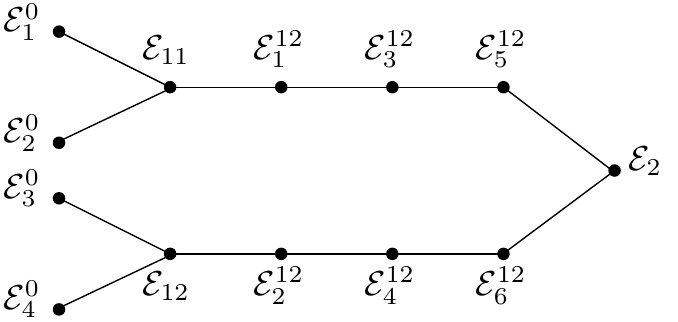}    
\caption{Dual graph of the good resolution of $(S,0) \subset (\C^4,0)$ defined by \eqref{eq:example}.}
\label{fig:dual-graph-example}
\end{figure}

\end{ex}

\section{Conditions for integral homology sphere link}\label{sec:proof}

In this section, we will prove the second part of Theorem~\ref{thm:rat-int-hom} for $g \geq 3$ using the good $\Q$-resolution $\hat{\varphi}:\hat{S} \rightarrow S$ of $(S,0)$ introduced in Section~\ref{sec:Q-resolution}. To this end, following Theorem~\ref{thm:conditions-rat-int}, we will investigate the determinant of $(S,0)$ with formula~\eqref{eq:det-bir-morphism} in terms of $\hat{\varphi}$.

\begin{remark} \label{rem:g=2}
\begin{enumerate}[wide,labelindent=0pt]
	\item[(i)] Note that Theorem~\ref{thm:rat-int-hom} generalizes the $g=2$ case or, thus, the classification for Brieskorn-Pham surface singularities in Proposition~\ref{prop:Brieskorn-Pham-rat-int}. Even more, for $g = 2$, one could also obtain this result by using the good $\Q$-resolution $\hat{\varphi} := \varphi_0: \hat{S} \rightarrow S$ of $(S,0)$.
	\item[(ii)] When the link of $(S,0)$ is a $\z$, we see that $r_k = 1$ and $N_k = n_k\lbeta_k$ for every $k = 1,\ldots, g-1$. Hence, all exceptional divisors $\E_k$ for $k = 1,\ldots, g-1$ are irreducible with multiplicity $n_k\lbeta_k$, and the dual graph of the good $\Q$-resolution $\hat{\varphi}: \hat{S} \rightarrow S$ is bamboo-shaped with quotient singularities as described in Section~\ref{sec:Q-resolution}. In particular, by Proposition~\ref{prop:rupture-divisors}, the good resolution of $(S,0)$ obtained from $\hat{\varphi}$ by resolving the singularities of $\hat{S}$ is minimal. 
\end{enumerate}
\end{remark}

\subsection{The determinant of the intersection matrix of the good $\Q$-resolution $\hat{\varphi}$} \label{sec:determinant}

Because we already know the singularities of $\hat{S}$, we will be able to compute the determinant of $(S,0)$ once we know the determinant of the intersection matrix $A$ of $\hat{\varphi}$. To compute the latter, we first need to calculate the (self-)intersection numbers of the exceptional curves $\E_{kj}$ for $k = 1,\ldots, g-1$ and $j = 1,\ldots,r_k$. Clearly, from the local situation~\eqref{eq:local-situation-Q_{k(k+1)}} around $Q_{k(k+1)}$ for every $k = 1,\ldots,g-2$, we immediately have $\E_{kj} \cdot \E_{(k+1)j'} = \frac{1}{d_{k(k+1)}}$ if $\E_{kj} \cap \E_{(k+1)j'} \neq \emptyset$. To find the self-intersection numbers $-a_k := \E_{kj}^2$, we can use the fact that $\hat{\varphi}^{\ast} Y \cdot \E_{kj} = 0$, see~\eqref{eq:intersection-pull-back}, where $\hat{\varphi}^{\ast}Y$ is given by~\eqref{eq:pull-back-Y}. Since $\hat{Y}$ only intersects $\E_{g-1}$ in the single point $P_{g-1}$ with local situation~\eqref{eq:local-situation-P_{g-1}}, we know that $\hat{Y} \cdot \E_{g-1} = \frac{n_g}{d}$ and $\hat{Y} \cdot \E_{kj} = 0$ for $k = 1,\ldots, g-2$ and $j = 1,\ldots, r_k$. We obtain
 \begin{equation}\renewcommand{\arraystretch}{1.6}\label{eq:self-intersection-numbers}
	a_k = \left\{\begin{array}{ll}
		\frac{N_2}{d_{12}N_1}  &\text{for } k = 1 \\
		\frac{1}{N_k}\Big(\frac{r_{k-1}N_{k-1}}{r_kd_{(k-1)k}} + \frac{N_{k+1}}{d_{k(k+1)}}\Big) & \text{for } k = 2,\ldots, g-2\\
		\frac{1}{N_{g-1}}\Big(\frac{r_{g-2}N_{g-2}}{d_{(g-2)(g-1)}} + \frac{n_g}{d}\Big) & \text{for } k = g-1.\\
\end{array}\right.
\end{equation} 
We can now write the intersection matrix $A$ as follows:
\begin{equation} \label{eq:A}
 A = \left(\begin{matrix} 
 A_1 & A_{1,2} & 0 & 0 &\cdots & 0  \\
 A_{2,1} & A_2 & A_{2,3} & 0 & \cdots & 0\\
 0 & A_{3,2} & A_3 & A_{3,4}  & \cdots & 0 \\
 \vdots & \vdots & \ddots & \ddots & \ddots & \vdots  \\
 0 & 0 & \cdots & A_{g-2,g-3} & A_{g-2} & A_{g-2,g-1} \\
 0 & 0 & \cdots & 0 &  A_{g-1,g-2} & A_{g-1}
 \end{matrix}\right).
\end{equation}
Here, we denote by $A_k$ for $k = 1,\ldots, g-1$ the $(r_k \times r_k)$-diagonal matrix with $-a_k$ on the diagonal, by $A_{k,k+1}$ for $k = 1,\ldots, g-2$ the $(r_k \times r_{k+1})$-matrix
\begin{equation}\label{eq:def-A_k,k+1}
	A_{k,k+1} = \left(
\begin{matrix}
 D_{k,k+1} & 0 & \cdots & 0 \\
 0 & D_{k,k+1}  &  \cdots & 0 \\
 \vdots &  \vdots & \ddots & \vdots  \\ 
  0 & 0  &  \cdots & D_{k,k+1}\\
  \end{matrix}\right),
\end{equation}
where $D_{k,k+1}$ is the $\frac{r_k}{r_{k+1}}$-column vector $(\frac{1}{d_{k(k+1)}},\ldots, \frac{1}{d_{k(k+1)}})^t$, and by $A_{k+1,k}  = A_{k,k+1}^t$ for $k = 1,\ldots, g-2$ the transpose of $A_{k,k+1}$. Note that $A_{g-1} = -a_g$ and $A_{g-2,g-1} = D_{g-2,g-1}$. \\

We will now show a formula for the determinant $\det(A)$ of a general matrix $A$ defined as in~\eqref{eq:A}. Hence, this formula can be used to compute the determinant of the intersection matrix for any good $\Q$-resolution with a dual graph as in Figure~\ref{fig:dual-graph}, in which the horizontally aligned exceptional curves are isomorphic, have the same self-intersection number, and have the same intersection behavior with the other exceptional curves. \\

We start by fixing some notation. First, for $k = 1,\ldots, g-2$, put $p_k := \frac{r_k}{r_{k+1}}$. Second, for $l = 2,\ldots, g-1$, let $s(l)$ be the set of non-empty subsets $K$ of ${\{(k,k+1) \mid k = 1,\ldots, l-1\}}$ such that for all $(k,k+1) \neq (k',k'+1) \in K$, we have $k \neq k'+1$ and $k' \neq k+1$. For such a set $K \in s(l)$, we call $c(K) :=\left\{ k \in \{1, \ldots, l\} \mid (k,k+1) \notin K, (k-1,k) \notin K\right\}$ its \emph{complement}. Finally, we introduce $R_0 := 1$, $R_1 := a_1$, and, for $l = 2,\ldots, g-1$, \[R_l := \prod_{k=1}^l a_k + \sum_{i=1}^{[\frac{l}{2}]} (-1)^i\sum_{K \in s(l), \vert K \vert = i} \bigg(\prod_{(k,k+1) \in K}\frac{p_k}{d_{k(k+1)}^2}\bigg)\bigg(\prod_{k \in c(K)} a_k \bigg),\] where we put $\prod_{k \in c(K)} a_k = 1$ if $c(K) = \emptyset$. For example, if $l = 2$, then $s(l)$ only contains the set $\{(1,2)\}$ with $c(\{(1,2)\}) = \emptyset$, so $R_2 = a_1a_2 - \frac{p_1}{d_{12}^2}.$ If $l = 3$, then $s(3)$ consists of two sets, $\{(1,2)\}$ and $\{(2,3)\}$, with complements $\{3\}$ and $\{1\}$, respectively. Hence, $R_3 = a_1a_2a_3 - \frac{p_1a_3}{d_{12}^2} - \frac{p_2a_1}{d_{23}^2}.$ \\

Before explaining how $\det(A)$ can be expressed in terms of these $R_l$ for $l = 1,\ldots, g-1$, we prove the following recurrence relation. 

\begin{lemma}\label{lem:recurrence-Rl}
	For all $l = 1,\ldots, g-2$, we have \[-R_{l+1} = -a_{l+1}R_l + \frac{p_lR_{l-1}}{d_{l(l+1)}^2}.\]
\end{lemma}

\begin{proof}
	For $l = 1$ and $l = 2$, this follows immediately from the simple expressions for $R_0$, $R_1$, $R_2$ and $R_3$. For $l \geq 3$, the right-hand side is by definition given by 
\begin{align*}
	- \Bigg(&\prod_{k = 1}^{l+1}a_k + \underbrace{\sum_{i=1}^{[\frac{l}{2}]} (-1)^i\sum_{K \in s(l), \vert K \vert = i} \bigg(\prod_{(k,k+1) \in K}\frac{p_k}{d_{k(k+1)}^2}\bigg)\bigg(a_{l+1}\prod_{k \in c(K)} a_k \bigg)}_{(a)} \\ 
 & \underbrace{- \frac{p_l}{d_{l(l+1)}^2}\prod_{k=1}^{l-1}a_k}_{(b)} + \underbrace{\sum_{i=1}^{[\frac{l-1}{2}]} (-1)^{i+1}\sum_{K \in s(l-1), \vert K \vert = i} \bigg(\frac{p_l}{d_{l(l+1)}^2} \prod_{(k,k+1) \in K}\frac{p_k}{d_{k(k+1)}^2}\bigg)\bigg(\prod_{k \in c(K)} a_k \bigg)}_{(c)}\Bigg).
\end{align*}
We need to show that (a) + (b) + (c) = (d) with \[(d) = \sum_{i=1}^{[\frac{l+1}{2}]} (-1)^i\sum_{K \in s(l+1), \vert K \vert = i} \bigg(\prod_{(k,k+1) \in K}\frac{p_k}{d_{k(k+1)}^2}\bigg)\bigg(\prod_{k \in c(K)} a_k \bigg).\] It is trivial that (b) corresponds to $K = \{(l,l+1)\}$ in (d). Using that $[\frac{l+1}{2}] = [\frac{l-1}{2}] + 1$, one can also see that (c) yields the part in (d) where $(l,l+1) \in K$ and $\vert K \vert \geq 2$. Hence, it remains to show that (a) corresponds to the part in (d) where $(l,l+1) \notin K$. Clearly, we only need to check that the boundaries for $\vert K \vert$ agree; in (a), the upper bound is $[\frac{l}{2}]$, while in (d), the upper bound is $[\frac{l+1}{2}]$. However, in (d), we need to take into account that $(l,l+1) \notin K$. We remark the following two facts:
\begin{enumerate}
	\item[(i)] if $l+1$ is even, then a set $K \in s(l+1)$ attains the upper bound $\vert K \vert = [\frac{l+1}{2}] = \frac{l+1}{2}$ if and only if $K = \{(1,2),(3,4),\ldots, (l,l+1)\}$; and 
	\item[(ii)] if $l+1$ is odd, then there are multiple sets in $s(l+1)$ attaining the upper bound $[\frac{l+1}{2}] = \frac{l}{2}$, for example $\{(1,2),(3,4),\ldots, (l-1,l)\}$ and $\{(2,3),(4,5), \ldots, (l,l+1)\}$.
\end{enumerate}
Hence, if $l+1$ is even, then $\vert K \vert$ for $K$ in (d) with $(l,l+1) \notin K$ varies between $1$ and $[\frac{l+1}{2}] - 1 = [\frac{l}{2}]$. In other words, the boundaries for $\vert K \vert$ agree. Likewise, if $l+1$ is odd, then $K$ in (d) with $(l,l+1)\notin K$ can still attain the upper bound $[\frac{l+1}{2}] = [\frac{l}{2}]$. 
\end{proof}

This recurrence relation will be very useful for showing the next formula for $\det(A)$. 

\begin{prop}\label{prop:detA}
Let $A$ be a matrix defined as in~\eqref{eq:A} for some $g \geq 3$, $r_k \geq 1$ for $k = 1,\ldots, g-1$ with $r_{g-1} = 1$, and $d_{k(k+1)} \geq 1$ for $k = 1,\ldots, g-2$. We have
	\[\det(A) = (-1)^{\sum\limits_{k=1}^{g-1}r_k} R_{g-1} \prod_{l=1}^{g-2}R_l^{r_l - r_{l+1}}.\]
\end{prop}

Using the recurrence relation from Lemma~\ref{lem:recurrence-Rl} and the expressions in~\eqref{eq:self-intersection-numbers} for $a_k$ for $k = 1,\ldots, g-1$ in which $\frac{r_{k-1}}{r_k} = p_{k-1}$, it is not hard to see that, in our case, the expression for $R_l$ simplifies to \[R_l = \left\{\renewcommand{\arraystretch}{1.5}\begin{array}{ll} \frac{N_{l+1}}{N_1\prod\limits_{k=1}^ld_{k(k+1)}} & \text{ for } l = 1,\ldots, g-2 \\ \frac{n_g}{N_1d\prod\limits_{k=1}^{g-2}d_{k(k+1)}} & \text{ for } l = g-1. \end{array} \right.\] This immediately yields the following expression for the determinant of the intersection matrix of the good $\Q$-resolution of our surface singularities.

\begin{cor}\label{cor:detA}
Let $(S,0) \subset (\C^{g+1},0)$ be a normal surface singularity defined by the equations~\eqref{eq:equations-S} with $g \geq 3$. Consider the good $\Q$-resolution $\hat{\varphi}:\hat{S} \rightarrow S$ of $(S,0)$ introduced in Section~\ref{sec:Q-resolution}. The determinant of the intersection matrix $A$ of $\hat{\varphi}$ is given by \[\det(A) =  (-1)^{\sum\limits_{k=1}^{g-1}r_k} ~\frac{n_g\prod\limits_{k=2}^{g-1}N_k^{r_{k-1}-r_k}}{N_1^{r_1} d\prod\limits_{k=1}^{g-2}d_{k(k+1)}^{r_k}}.\]
\end{cor}

In order to better understand the idea of the proof of Proposition~\ref{prop:detA}, we first consider the simple case where $r_k = 1$ for all $k = 1,\ldots, g-1$, and $A$ is the tridiagonal matrix 
\[\left(\begin{matrix}
 -a_1 &\frac{1}{d_{12}} & 0 &\cdots & 0  \\
 \frac{1}{d_{12}} & -a_2 & \frac{1}{d_{23}}  & \cdots & 0\\
 \vdots & \ddots & \ddots & \ddots & \vdots  \\
 0 & \cdots & \frac{1}{d_{(g-3)(g-2)}} & -a_{g-2} & \frac{1}{d_{(g-2)(g-1)}} \\
 0 & \cdots & 0 &  \frac{1}{d_{(g-2)(g-1)}} & -a_{g-1} \end{matrix}
\right).\]
If we denote this matrix for a moment by $A(g)$ for $g \geq 3$, then the general three-term recurrence relation for the determinant of tridiagonal matrices tells us that 
\begin{equation}\label{eq:recurrence-tridiagonal}
 	\det(A(g)) = -a_{g-1}\det(A(g-1)) - \frac{1}{d_{(g-2)(g-1)}^2}\det(A(g-2)),
 \end{equation}
where, by convention, we put $A(1) = 1$ and $A(2) = (-a_1)$. This recurrence relation can be shown by first expanding the determinant of $A(g)$ along the last column (resp. row) and then expanding the minor corresponding to $\frac{1}{d_{(g-2)(g-1)}}$ along the last row (resp. column). Note the similarity between this relation and the relation from Lemma~\ref{lem:recurrence-Rl}. Even more, by induction on $g$ and with exactly the same argument as in the proof of Lemma~\ref{lem:recurrence-Rl}, one can show that $\det(A(g)) = (-1)^{g-1}R_{g-1}$ for $g \geq 3$, in which $p_k = 1$ for all $k = 1,\ldots, g-2$. In other words, the recurrence relation satisfied by the $R_l$ for $l = 1, \ldots, g-2$ in Lemma~\ref{lem:recurrence-Rl} is a generalization of~\eqref{eq:recurrence-tridiagonal} by allowing general $p_k \geq 1$ for $k = 1,\ldots, g-2$. \\

To show Proposition~\ref{prop:detA} for general $r_k \geq 1$ for $k = 1,\ldots, g-2$, we will work towards tridiagonal matrices of the following type: 
\begin{equation}\label{eq:B}
 	B_s := \left(\begin{matrix} 
	-a_s & \frac{1}{d_{s(s+1)}} & \cdots & 0 \\
	\frac{1}{d_{s(s+1)}} & \ddots & \ddots & \vdots \\
	\vdots & \ddots &  \ddots &  \frac{1}{d_{(g-2)(g-1)}} \\
	0 & \cdots & \frac{1}{d_{(g-2)(g-1)}} & -a_{g-1} 
 	\end{matrix}\right),
\end{equation}
where $s \in \{1,\ldots, g-1\}$. Note that $A(g) = B_1$ and, thus, that $\det(B_1) = (-1)^{g-1}R_{g-1}$. For general $s$, we can write the determinant of $B_s$ as $(-1)^{g-s}R_{g-s}$ in which we start with $a_s$ instead of $a_1$. We will write $\det(A)$ (for $g \geq 4$) in terms of these tridiagonal matrices using the formula in the next result.
  
\begin{lemma}\label{lem:property-t-and-R_l}
Consider $g \geq 4$. Let $t$ be the smallest $k \in \{1,\ldots, g-1\}$ such that $r_k = 1$. Assume that $2 \leq t \leq g-2$. Then, \[R_{t-1}\det(B_t) + \frac{p_{t-1}R_{t-2}}{d_{(t-1)t}^2}\det(B_{t+1}) = (-1)^{g-t}R_{g-1}.\]  	
\end{lemma}

\begin{proof}
First, note that such $t \in \{1,\ldots, g-1\}$ always exists as $r_{g-1} = 1$. Furthermore, note that $r_k = 1$ for all $k \geq t$ so that $p_k = 1$ for all $k \geq t$. With the expression for $\det(B_t)$ (resp. $\det(B_{t+1})$) in terms of $R_{g-t}$ (resp. $R_{g - t- 1}$) in which we start with $a_t$ (resp. $a_{t+1}$) instead of $a_1$ and all $p_k = 1$, we can show this formula with similar arguments as in the proof of Lemma~\ref{lem:recurrence-Rl}. However, we will prove the stronger result that  \[R_{s-1}\det(B_s) + \frac{p_{s-1}R_{s-2}}{d_{(s-1)s}^2}\det(B_{s+1}) = (-1)^{g-s}R_{g-1}\] for all $s = t, \ldots, g-2$ by using backward induction and the statement of Lemma~\ref{lem:recurrence-Rl}. 
For $s = g-2$, we need to consider 
\begin{align*}
 & R_{g-3}\det \left(\begin{matrix} 
-a_{g-2} &  \frac{1}{d_{(g-2)(g-1)}} \\
\frac{1}{d_{(g-2)(g-1)}} & -a_{g-1} 
 \end{matrix}\right) + \frac{p_{g-3}R_{g-4}}{d_{(g-3)(g-2)}^2}\det(-a_{g-1}) \\
 & = -a_{g-1}\bigg(-a_{g-2}R_{g-3} + \frac{p_{g-3}R_{g-4}}{d_{(g-3)(g-2)}^2}\bigg) - \frac{R_{g-3}}{d_{(g-2)(g-1)}^2},
 \end{align*}
 and show that this is equal to $(-1)^{g-s}R_{g-1} = R_{g-1}$. This follows from first applying Lemma~\ref{lem:recurrence-Rl} for $l = g-3$ and then for $l = g-2$ with $p_{g-2} = 1$. If $t = g-2$, we are done. Otherwise, suppose it is true for $s + 1 \leq g-2$. For $s$, we first expand $\det(B_s)$ along the first column and then expand the second minor along the first row to get 
 \[ R_{s-1}\det(B_s) + \frac{p_{s-1}R_{s-2}}{d_{(s-1)s}^2}\det(B_{s+1}) = \Big(-a_s R_{s-1} + \frac{p_{s-1}R_{s-2}}{d_{(s-1)s}^2}\Big) \det(B_{s+1}) - \frac{R_{s-1}}{d_{s(s+1)}^2}\det(B_{s+2}).\]
This way of rewriting $\det(B_s)$ is the same as the one we can use to show the three-term recurrence relation~\eqref{eq:recurrence-tridiagonal} for the tridiagonal matrices $A(g)$, but with expansion along the first column instead of along the last column. Because of the similarity between the relations in~\eqref{eq:recurrence-tridiagonal} and~Lemma~\ref{lem:recurrence-Rl}, it is no surprise that we can apply Lemma~\ref{lem:recurrence-Rl} for $l = s-1$ so that \[R_{s-1}\det(B_s) + \frac{p_{s-1}R_{s-2}}{d_{(s-1)s}^2}\det(B_{s+1}) = -R_s\det(B_{s+1}) - \frac{R_{s-1}}{d_{s(s+1)}^2}\det(B_{s+2}).\] Since $p_s = 1$ as $s \geq t$, we can conclude with the induction hypothesis.
\end{proof}

We are now ready to prove Proposition~\ref{prop:detA} by using these matrices $B_s$. 

\begin{proof}[Proof of Proposition~\ref{prop:detA}]
As in the previous lemma, let $t$ be the smallest $k \in \{1,\ldots, g-1\}$ such that $r_k = 1$. If $t = 1$, we already know that $\det(A) = (-1)^{g-1}R_{g-1}$. For $t \geq 2$, we will show that \[\det(A) = (-1)^{g-t +\sum\limits_{k=1}^{t-1}r_k} R_{g-1} \prod_{l=1}^{t-1}R_l^{r_l - r_{l+1}}.\] Because $r_k = 1$ for $k \geq t$, this yields the formula given in the proposition. Throughout the proof, we will denote by $A(r_1,\ldots, r_{g-1})$ a matrix defined as in~\eqref{eq:A} corresponding to some $r_1,\ldots, r_{g-1} \geq 1$ with $g \geq 3$ in which we also allow $r_{g-1} > 1$. To get an idea on how to show the above formula for general $t$, we first consider $t=2, t = 3$ and $t = 4$.\\

If $t = 2$, then $A = A(r_1,1,\ldots, 1)$ with $r_1 \geq 2$. If $g\geq 4$, we can, similarly as in the proof of Lemma~\ref{lem:property-t-and-R_l}, first expand $\det(A)$ along the first column and then expand the minor corresponding to $\frac{1}{d_{12}}$ along the first row to find that
\begin{align*}
	\det(A) & = -a_1\det(A(r_1-1,1,\ldots, 1)) -\frac{1}{d_{12}^2}\det \left(\begin{matrix}
	A_1^{r_1-1} & 0 \\
	0 & B_3
	\end{matrix}\right) \\
	& = -a_1\det(A(r_1-1,1,\ldots, 1)) + \frac{(-1)^{r_1}a_1^{r_1-1}}{d_{12}^2} \det(B_3),
\end{align*}
where $A_1^{r_1-1}$ denotes the diagonal matrix of dimension $r_1-1$ with $-a_1$ on its diagonal. We can now repeat this on $\det(A(r_1-1,1,\ldots, 1))$: we expand the determinant along the first column and simplify the minor corresponding to $\frac{1}{d_{12}}$. This yields 
\[\det(A) = a_1^2\det(A(r_1-2,1,\ldots, 1)) + \frac{2(-1)^{r_1}a_1^{r_1-1}}{d_{12}^2} \det (B_3).\] Note that the first determinant for $r_1 = 2$ is just $\det(B_2)$. If we do this procedure $r_1 = p_1$ times in total, we get 
\begin{align*}
	\det(A) & = (-1)^{r_1}a_1^{r_1}\det(B_2) + \frac{r_1(-1)^{r_1}a_1^{r_1-1}}{d_{12}^2}\det(B_3) \\
	& = (-1)^{r_1}R_1^{r_1-1}\bigg(R_1\det(B_2) + \frac{p_1R_0}{d_{12}^2}\det(B_3)\bigg)\\
	& = (-1)^{r_1 + g-2} R_1^{r_1-1}R_{g-1}, 
\end{align*}
where we applied Lemma~\ref{lem:property-t-and-R_l} in the last equality. If $g = 3$, then along the same lines, we obtain that \[\det(A) = (-1)^{r_1}R_1^{r_1-1}\bigg(-a_2R_1 + \frac{p_1R_0}{d_{12}^2}\bigg),\] from which the required formula follows by Lemma~\ref{lem:recurrence-Rl}. \\

If $t = 3$ and $g\geq 5$, we start by executing two steps. In the first step, we work as in the $t = 2$ case: $p_1$ times in total, we first expand along the first column and then expand the second minor once more along the first row. This way, we can rewrite $\det(A) = \det(A(r_1,r_2,1,\ldots, 1))$ as \[	(-1)^{p_1}R_1^{p_1-1}\bigg(R_1 \det(\tilde{A}(r_1-p_1,r_2,1,\ldots,1) + \frac{p_1R_0}{d_{12}^2}\det(A(r_1-p_1,r_2-1,1,\ldots,1)\bigg),\] where $\tilde{A}(r_1-p_1,r_2,1,\ldots,1)$ is the matrix
\[\left(\begin{matrix} 
  A_1^{r_1-p_1} & [0 \mid A_{1,2}^{r_1-p_1, r_2-1}] & 0 & 0 & \cdots & 0\\
  [0 \mid A_{1,2}^{r_1-p_1, r_2-1}]^t & A_2 & A_{2,3} & 0 & \cdots & 0 \\
  0 & A_{3,2} & -a_3 & \frac{1}{d_{34}} & \cdots & 0 \\
  0 & 0 & \frac{1}{d_{34}} &  \ddots & \ddots & \vdots\\
  \vdots & \vdots  & \vdots & \ddots &  \ddots & \frac{1}{d_{(g-2)(g-1)}} \\
  0 &  0 & 0 &  \cdots & \frac{1}{d_{(g-2)(g-1)}} & -a_{g-1} 
 \end{matrix}\right),\]
in which $A_1^{r_1-p_1}$ denotes the diagonal matrix of dimension $r_1-p_1$ with $-a_1$ on its diagonal, $A_{1,2}^{r_1-p_1,r_2-1}$ denotes the $(r_1-p_1)\times (r_2-1)$-matrix defined in terms of the column vector $D_{1,2} = (\frac{1}{d_{12}},\ldots, \frac{1}{d_{12}})^t$ of length $p_1$ as in \eqref{eq:def-A_k,k+1}, and $[0\mid A_{1,2}^{r_1-p_1,r_2-1}]$ is the $(r_1-p_1) \times r_2$-matrix coming from $A_{1,2}^{r_1-p_1,r_2-1}$ by adding a zero column. In the second step, we expand $\det(\tilde{A}(r_1-p_1,r_2,1,\ldots,1)$) along the $(r_1-p_1 + 1)$th column (i.e. the column corresponding to the first entry of $A_2$, which also contains the zero column of $[0\mid A_{1,2}^{r_1-p_1,r_2-1}]$) and simplify the minor corresponding to $\frac{1}{d_{23}}$. We find that $\det(A)$ is given by
\begin{align*}
  (-1)^{p_1}R_1^{p_1-1}\Bigg[&\bigg(-a_2R_1 + \frac{p_1R_0}{d_{12}^2}\bigg)\det(A(r_1-p_1,r_2-1,1,\ldots,1)) \\
  & - \frac{R_1}{d_{23}^2} \det(A(r_1-p_1,r_2-1))\det(B_4)\Bigg].
 \end{align*}
By Lemma~\ref{lem:recurrence-Rl}, this is equal to \[
  (-1)^{p_1+1}R_1^{p_1-1}\Bigg[R_2\det(A(r_1-p_1,r_2-1,1,\ldots,1)) + \frac{R_1}{d_{23}^2} \det(A(r_1-p_1,r_2-1))\det(B_4)\Bigg].\]
Repeating both steps on $\det(A(r_1-p_1,r_2-1,1,\ldots,1))$ and $\det((A(r_1-p_1,r_2-1))$ gives 
\begin{align*}
	\det(A) = (-1)^{2(p_1+1)}R_1^{2(p_1-1)}R_2 \Bigg[&R_2\det(A(r_1-2p_1,r_2-2,1,\ldots, 1) \\
	& + \frac{2R_1}{d_{23}^2}\det(A(r_1-2p_1,r_2-2))\det(B_4)\Bigg]. 
\end{align*} 
Note that for $\det(A(r_1-p_1,r_2-1))$, we do not have a minor corresponding to $\frac{1}{d_{23}}$ in the second step. Hence, if we do these two steps $r_2 = p_2$ times in total, we find that 
\[\det(A) = (-1)^{(p_1+1)r_2}R_1^{(p_1-1)r_2}R_2^{r_2-1} \bigg(R_2\det(B_3) + \frac{p_2R_1}{d_{23}^2}\det(B_4)\bigg).\] We can conclude using Lemma~\ref{lem:property-t-and-R_l} and the fact that $r_1 = p_1r_2$. The result for $t = 3$ and $g = 4$ again follows along the same lines with Lemma~\ref{lem:recurrence-Rl}. \\

For $t = 4$ and $g \geq 6$, we can compute $\det(A) = \det(A(r_1,r_2, r_3,1,\ldots, 1))$ as follows. We first follow the procedure that we used for $t = 3$. More precisely, we execute $p_2$ times two steps: first, we expand $p_1$ times along the first column, and then, we expand along the column corresponding to the first entry of $A_2$, and in both steps, we simplify the second minor corresponding to $\frac{1}{d_{12}}$ and $\frac{1}{d_{23}}$, respectively. In other words, we rewrite $\det(A)$ as
\begin{align*}
 (-1)^{(p_1+1)p_2}R_1^{(p_1-1)p_2}R_2^{p_2-1} \bigg(& R_2\det(\tilde{A}(r_1-p_1p_2,r_2-p_2,r_3,1,\ldots, 1)) \\
 	&+ \frac{p_2R_1}{d_{23}^2}\det(A(r_1-p_1p_2,r_2-p_2,r_3-1,1,\ldots, 1))\bigg),
 \end{align*}
where $\tilde{A}(r_1-p_1p_2,r_2-p_2,r_3,1,\ldots, 1)$ is the matrix 
\[\left(\begin{matrix} 
  A_1^{r_1-p_1p_2} & A_{1,2}^{r_1-p_1p_2, r_2-p_2} & 0 & 0 & \cdots & 0\\
  (A_{1,2}^{r_1-p_1p_2, r_2-p_2})^t & A_2^{r_2-p_2} & [0 \mid A_{2,3}^{r_2-p_2,r_3-1}] & 0 & \cdots & 0 \\
  0 & [0 \mid A_{2,3}^{r_2-p_2,r_3-1}]^t & A_3 & A_{3,4} & \cdots & 0 \\
  0 & 0 & A_{4,3} &  \ddots & \ddots & \vdots\\
  \vdots & \vdots  & \vdots & \ddots &  \ddots & \frac{1}{d_{(g-2)(g-1)}} \\
  0 &  0 & 0 &  \cdots & \frac{1}{d_{(g-2)(g-1)}} & -a_{g-1} 
 \end{matrix}\right),\] in which we use the same notation as before. Now, by expanding along the column containing the first entry of $A_3$, simplifying the minor of $\frac{1}{d_{34}}$ and using Lemma~\ref{lem:recurrence-Rl}, we can further rewrite $\det(A)$ as 
 \begin{align*}
 	(-1)^{(p_1+1)p_2+1}R_1^{(p_1-1)p_2}R_2^{p_2-1}\Bigg[& R_3 \det(A(r_1-p_1p_2,r_2-p_2,r_3-1,1,\ldots, 1)) \\
 	& + \frac{R_2}{d_{34}^2} \det(A(r_1-p_1p_2,r_2-p_2,r_3-1))\det(B_5)\Bigg].
 \end{align*}
 We can repeat these two steps (i.e., the procedure for $t = 3$ followed by an expansion along the column corresponding to the first entry of $A_3$) on $\det(A(r_1-p_1p_2,r_2-p_2,r_3-1,1,\ldots, 1))$ and $\det(A(r_1-p_1p_2,r_2-p_2,r_3-1))$. In total, we can do this $r_3 = p_3$ times to find that \[\det(A) = (-1)^{((p_1+1)p_2+1)r_3}R_1^{(p_1-1)p_2r_3}R_2^{(p_2-1)r_3}R_3^{r_3-1}\bigg(R_3 \det(B_4) + \frac{p_3R_2}{d_{34}^2} \det(B_5)\bigg),\]
 which equals the required formula by Lemma~\ref{lem:property-t-and-R_l}. The case $g = 5$ can once more be concluded along the same lines. \\
 
For general $t \geq 3$ and $g \geq t + 2$, we can obtain the above formula for $\det(A)$ in a similar way as for $t = 3$ and $t = 4$. More precisely, we first repeat the procedure used for $t - 1$ to obtain an expression involving a matrix similar to $\tilde{A}(r_1-p_1,r_2-1,1,\ldots, 1)$ and $\tilde{A}(r_1-p_1p_2,r_2-p_2,r_3,1,\ldots, 1)$. Then, we can further expand along the column containing the first entry of $A_{t-1}$, simplify the minor of $\frac{1}{d_{(t-1)t}}$ and use Lemma~\ref{lem:recurrence-Rl}. Again, executing these two steps $p_t$ times in total, yields \[\det(A) = (-1)^{\sum_{k=1}^tr_k}\prod_{l=1}^{t-1}R_l^{r_l-r_{l+1}}\bigg(R_{t-1}\det(B_t) + \frac{p_{t-1}R_{t-2}}{d_{(t-1)t}^2}\det(B_{t+1})\bigg),\] from which the formula follows with Lemma~\ref{lem:property-t-and-R_l}. If $g = t+1$, then the formula follows along the same lines with Lemma~\ref{lem:recurrence-Rl}.
\end{proof}

\subsection{The determinant of $(S,0)$} With the information on the singularities of $\hat{S}$ that we listed in Section~\ref{sec:Q-resolution} and the expression for $\det(A)$ from Corollary~\ref{cor:detA}, we immediately find the determinant of $(S,0)$; it is given by 
\begin{align*}
\det(S) & = \vert\det(A) \vert ~ d ~ \left(\frac{N_1}{M_0}\right)^{\frac{\lbeta_0}{M_0}} ~ \prod_{k=1}^{g-1}\left(\frac{N_k}{M_k}\right)^{\frac{\lbeta_k}{M_k}} ~ \prod_{k=1}^{g-2}d_{k(k+1)}^{r_k} \\
& = \left(\frac{N_1}{M_0} \right)^{\frac{\lbeta_0}{M_0}-r_1} ~ \prod_{k=1}^{g-1}\left(\frac{N_k}{M_k} \right)^{\frac{\lbeta_k}{M_k}-r_k} ~ \prod_{k=2}^{g-1}N_k^{r_{k-1}-r_k} ~ n_g ~ \left(\frac{1}{M_0} \right)^{r_1} ~ \prod_{k=1}^{g-1}\left(\frac{N_k}{M_k} \right)^{r_k}. 
\end{align*}

From the expression~\eqref{eq:d_k} for $d_k = \frac{N_k}{M_k}$ for $k = 1,\ldots, g-1$, we know that \[\frac{N_k}{M_k} = \frac{\lcm(n_k,\ldots, n_g)}{\lcm(n_{k+1},\ldots, n_g)}.\] Note that for $k = 1$, this gives that $\frac{N_1}{M_1} = \frac{M_0}{\lcm(n_2,\ldots, n_g)}$. Hence, using the notation $r_0 = \frac{\lbeta_0}{M_0}$, we can further rewrite $\det(S)$ into the following expression.

\begin{cor}\label{cor:detS}
The determinant of a normal surface singularity $(S,0) \subset (\C^{g+1},0)$ defined by the equations~\eqref{eq:equations-S} with $g \geq 3$ is given by 
\[\det(S) = \prod_{k=1}^{g-1}\left(\frac{N_k}{M_k} \right)^{\frac{\lbeta_k}{M_k}-r_k} \left(\frac{N_k}{\lcm(n_k,\ldots, n_g)}\right)^{r_{k-1}-r_k}.\]
\end{cor}

According to Theorem~\ref{thm:conditions-rat-int}, we need to investigate when this determinant is equal to $1$, under the condition that the link of $(S,0)$ is already a $\q$ or, in other words, that $\gcd(n_k,\lcm(n_{k+1},\ldots, n_g)) = 1$ or $\gcd(\frac{\lbeta_k}{e_k},\lcm(n_{k+1},\ldots, n_g)) =1$ for all ${k = 1,\ldots, g-1}$. Recall that the condition $\gcd(n_k,\lcm(n_{k+1},\ldots, n_g)) = 1$ is equivalent to $r_{k-1} = r_k$. Furthermore, it is equivalent to $\frac{N_k}{M_k} = n_k$. In other words, if $\gcd(n_k,\lcm(n_{k+1},\ldots, n_g)) = 1$, then the part for $k$ in $\det(S)$ is given by \[\left(\frac{N_k}{M_k} \right)^{\frac{\lbeta_k}{M_k}-r_k} \left(\frac{N_k}{\lcm(n_k,\ldots, n_g)}\right)^{r_{k-1}-r_k} = n_k^{\frac{\lbeta_k}{M_k}-r_k}.\] Similarly, the condition $\gcd(\frac{\lbeta_k}{e_k},\lcm(n_{k+1},\ldots, n_g)) =1$ is equivalent to both $\frac{\lbeta_k}{M_k} = r_k$ and $\frac{N_k}{\lcm(n_k,\ldots, n_g)} = \frac{\lbeta_k}{e_k}$ so that in this case, the part for $k$ is given by \[\frac{\lbeta_k}{e_k}^{r_{k-1}-r_k}.\] This implies that, in both cases, the part for $k$ in $\det(S)$ is equal to $1$ if and only if $\gcd(n_k,\lcm(n_{k+1},\ldots, n_g)) = \gcd(\frac{\lbeta_k}{e_k},\lcm(n_{k+1},\ldots, n_g)) =1$. It follows that $\det(S)$ is equal to $1$ if and only if $\gcd(n_k,\lcm(n_{k+1},\ldots, n_g)) = \gcd(\frac{\lbeta_k}{e_k},\lcm(n_{k+1},\ldots, n_g)) =1$ for all $k = 1,\ldots, g-1$. Finally, one can see that the condition that $\gcd(n_k,\lcm(n_{k+1},\ldots, n_g)) = 1$ for all $k = 1,\ldots, g-1$ is equivalent to the condition that $n_i$ for $i = 1,\ldots, g$ are pairwise coprime. Hence, the condition $\gcd(\frac{\lbeta_1}{e_1},\lcm(n_2,\ldots, n_g)) = 1$ becomes $\gcd(n_0,n_2,\ldots, n_g) = 1$, which is equivalent to $\gcd(n_0,n_i) = 1$ for all $i =2,\ldots, g$. Because $n_0$ and $n_1$ are coprime by assumption, we indeed find that $(S,0)$ has a $\z$ link if and only if the exponents $n_i$ for $i = 0,\ldots, g$ are pairwise coprime and $\gcd(\frac{\lbeta_k}{e_k},n_{k+1}\cdots n_g) = \gcd(\frac{\lbeta_k}{e_k},e_k) = 1$ for $k = 2,\ldots, g-1$. This ends our proof of Theorem~\ref{thm:rat-int-hom}. 

\begin{ex}\label{ex:int-hom}
Consider the surface $S_1 \subset \C^4$ ($g = 3$) defined by the equations
	\[\renewcommand{\arraystretch}{1.2}{\left\{\begin{array}{l c l l l}
	x_1^2-x_0^3 & + & x_2^7 - x_0^{20}x_1 &  = 0 \\
	x_2^7 - x_0^{20}x_1 & + & x_3^5 - x_0^{88}x_1x_2^6 &= 0.
    \end{array}\right.}\]	
The semigroup of the corresponding space monomial curve has $(70,105,215,1511)$ as minimal generating set. By Theorem~\ref{thm:rat-int-hom}, the link of $(S,0)$ is a $\z$ as the exponents $3,2,7$ and $5$ are pairwise coprime and $\gcd(\frac{\lbeta_2}{e_2},e_2) = \gcd(\frac{215}{5},5) = 1$. However, if we modify these equations slightly, then the surface $S_2 \subset \C^4$ given by \[\renewcommand{\arraystretch}{1.2}{\left\{\begin{array}{l c l l l}
	x_1^2-x_0^3 & + & x_2^7 - x_0^{21}x_1 &  = 0 \\
	x_2^7 - x_0^{21}x_1 & + & x_3^5 - x_0^{92}x_1x_2^6 &= 0
    \end{array}\right.}\] 
does not have a $\z$ link. Indeed, the corresponding set of generators is $(70,105,225,1579)$ with $\gcd(\frac{\lbeta_2}{e_2},e_2) = \gcd(\frac{225}{5},5) \neq 1$. Note that the link of $(S_2,0)$ is a $\q$ as the exponents $3,2,7$ and $5$ are still pairwise coprime. The surface singularity from Example~\ref{ex:E_{g-1}-superfluous} is an example of a surface singularity in our family with no pairwise coprime exponents, but whose link is a $\q$ as $\gcd(\frac{\lbeta_1}{e_1},\lcm(n_2,n_3)) = \gcd(\frac{12}{4},2) = 1$ and $ \gcd(\frac{\lbeta_2}{e_2},n_3) = \gcd(\frac{26}{2},2) = 1$. Finally, the equations \[\renewcommand{\arraystretch}{1.2}{\left\{\begin{array}{l c l l l}
	x_1^2-x_0^3 & + & x_2^4 - x_0^{11}x_1 &  = 0 \\
	x_2^4 - x_0^{11}x_1 & + & x_3^3 - x_0^{28}x_1x_2^3 &= 0
    \end{array}\right.}\] 
    define a surface $S_3 \subset \C^4$ with neither a $\q$ nor a $\z$ link: the corresponding generating set is $(24,36,75,311)$ with $\gcd(n_1,\lcm(n_2,n_3)) = \gcd(2,12) \neq 1$ and $\gcd(\frac{\lbeta_1}{e_1},\lcm(n_2,n_3)) = \gcd(\frac{36}{12},12) \neq 1$.
    \end{ex}
    
\subsection{Our surface singularities with $\z$ link versus singularities of splice type}\label{sec:our-surfaces-splice-type}

We finish this article by showing that if $(S,0)$ has a $\z$ link, then it is of splice type. In other words, they belong to the family of complete intersection singularities of splice type defined by Neumann and Wahl and support their conjecture on the possible normal complete intersection surface singularities with a $\z$ link. \\

Since $(S,0)$ for $g = 2$ is trivially of splice type, we assume that $g \geq 3$. We first determine the splice diagram of $(S,0)$. We can again use the good $\Q$-resolution $\hat{\varphi}:\hat{S} \rightarrow S$. In Remark~\ref{rem:g=2}, we already mentioned that each exceptional divisor $\E_k$ for $k = 1,\ldots, g-1$ is irreducible with multiplicity $N_k = n_k\lbeta_k$, and that the dual graph is bamboo-shaped with quotient singularities as described in Section~\ref{sec:Q-resolution}. Taking a closer look at these singularities, one can check that the resolution $\hat{\varphi}$ is as in Figure~\ref{fig:res-integral}, where the numbers in brackets represent the orders of the small groups acting on the singular points.

\begin{figure}[ht]
\includegraphics{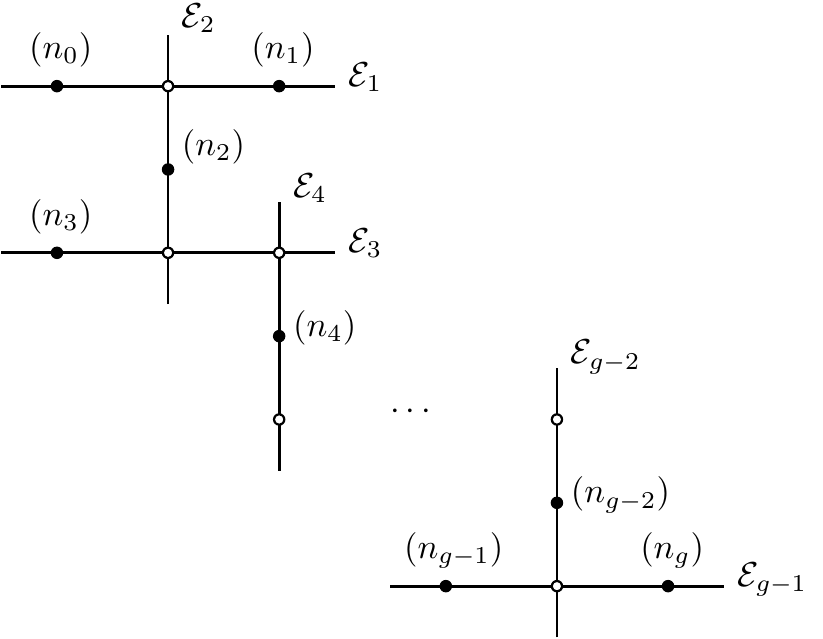}
\caption{The good $\Q$-resolution of $(S,0)$ when the link is a $\z$.}
\label{fig:res-integral}
\end{figure}

It immediately follows that the splice diagram is of the form as in Figure~\ref{fig:splice-diagram-integral}, in which the nodes from left to right correspond to $\E_k$ for $k = 1,\ldots, g-1$, and the edge weights $n_k$ for $k =0,\ldots, g$ come from the singular points $Q_k$ for $k = 0,\ldots, g-1$ and $P_{g-1}$. It remains to show that the other weights are given as in the figure. \\

\begin{figure}[ht]
\includegraphics{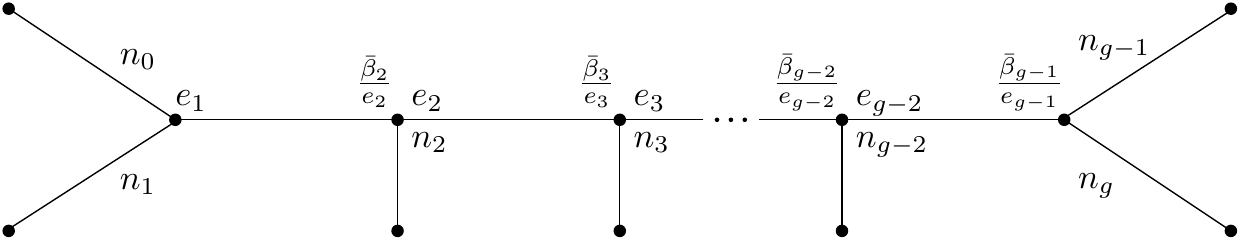}
\caption{The splice diagram of $(S,0)$ with a $\z$ link}
\label{fig:splice-diagram-integral}
\end{figure}

We start by showing that the order $d_{k(k+1)}$ corresponding to $Q_{k(k+1)} = \E_k \cap \E_{k+1}$ for $k = 1,\ldots, g-2$ becomes very easy. Following the approach explained in Section~\ref{sec:Q-resolution}, we need to consider the quotient space 
\[X \left(\begin{array}{c|cc}
			n_{k+1} \bar{\beta}_{k+1} - n_k \bar{\beta}_k & e_k & -e_k \\[0.2cm]
			n_{k+1} \bar{\beta}_{k+1} - n_k \bar{\beta}_k & - \frac{\bar{\beta}_{k+1}}{e_{k+1}} & \frac{n_k\lbeta_k}{e_k}
			\end{array} \right).\] 
Because $\gcd(\frac{\lbeta_{k+1}}{e_{k+1}},e_{k+1}) = 1$ by assumption on the link, and $\gcd(\frac{\lbeta_{k+1}}{e_{k+1}},n_{k+1}) = 1$ by the properties of the semigroup, we see that $\gcd(\frac{\lbeta_{k+1}}{e_{k+1}},e_k) = \gcd(\frac{\lbeta_{k+1}}{e_{k+1}},n_{k+1}e_{k+1}) = 1$. Hence, the isomorphism in \eqref{eq:upper-triangular} says that this quotient space is isomorphic to 
\[X \left(\begin{array}{c|cc}
			n_{k+1} \bar{\beta}_{k+1} - n_k \bar{\beta}_k & 1 & - \alpha e_k + \beta\frac{n_k\lbeta_k}{e_k} \\[0.2cm]
			n_{k+1} \bar{\beta}_{k+1} - n_k \bar{\beta}_k & 0 &  n_k \bar{\beta}_k - n_{k+1} \bar{\beta}_{k+1} 
			\end{array} \right) = X\Big(n_{k+1} \bar{\beta}_{k+1} - n_k \bar{\beta}_k ; 1 , -\alpha e_k + \beta\frac{n_k\lbeta_k}{e_k}\Big),\] 
where $\alpha e_k - \beta \frac{\lbeta_{k+1}}{e_{k+1}} = 1$. It follows that $d_{k(k+1)} = n_{k+1}\lbeta_{k+1} - n_k\lbeta_k = N_{k+1} - N_k$. Note that this is consistent with~\eqref{eq:d_{k(k+1)}}. Using this, one can also easily see that the expressions~\eqref{eq:self-intersection-numbers} for the self-intersection numbers $-a_k := \E_{kj}^2$ become 
\[a_k = \renewcommand{\arraystretch}{1.5} \left\{\begin{array}{ll}
		\frac{N_2}{d_{12}N_1}  &\text{for } k = 1 \\
		\frac{N_{k+1} - N_{k-1}}{d_{(k-1)k}d_{k(k+1)}} & \text{for } k = 2,\ldots, g-2\\
		\frac{1}{d_{(g-2)(g-1)}}  & \text{for } k = g-1.\\
\end{array}\right.\]

Let us now take a look at the edge to the right of $\E_1$. To show that its weight is $e_1$, we need to compute the determinant of the intersection matrix corresponding to the dual graph coming from removing $\E_1$ in Figure~\ref{fig:res-integral} and resolving the singular points on $\E_2,\ldots, \E_{g-1}$. By~\eqref{eq:det-bir-morphism}, this is equal to \[\vert\det(B_2)\vert~\prod_{l=2}^gn_l~\prod_{l=1}^{g-2}d_{l(l+1)} = \vert\det(B_2)\vert~e_1~\prod_{l=1}^{g-2}d_{l(l+1)},\] where $B_2$ is defined as in~\eqref{eq:B}. Hence, we need to check that \[\vert\det(B_2)\vert =  \frac{1}{\prod_{l=1}^{g-2}d_{l(l+1)}}.\] Similarly, for $k = 2,\ldots, g-2$ (if $g\geq 4$), we want that \[e_k = \vert\det(B_{k+1})\vert~ \prod_{l=k+1}^gn_l~\prod_{l=k}^{g-2}d_{l(l+1)} = \vert\det(B_{k+1})\vert~ e_k ~\prod_{l=k}^{g-2}d_{l(l+1)}.\] Using the expression for $\det(B_s)$ for $s = 2,\ldots, g-1$ in terms of the $R_l$ from Section~\ref{sec:determinant}, or using an induction argument, one can see that this is indeed true. \\

Analogously, to show that the weight on the edge to the left of $\E_k$ for ${k = 2,\ldots, g-1}$ is equal to $\frac{\lbeta_k}{e_k}$, we need to check that \[\frac{\lbeta_k}{e_k} = \vert\det(B'_{k-1})\vert~\prod_{l=0}^{k-1}n_l~\prod_{l=1}^{k-1}d_{l(l+1)},\] where \[B'_s := \left(\begin{matrix} 
	-a_1 & \frac{1}{d_{12}} & \cdots & 0 \\
	\frac{1}{d_{12}} & \ddots & \ddots & \vdots \\
	\vdots & \ddots &  \ddots &  \frac{1}{d_{(s-1)s}} \\
	0 & \cdots & \frac{1}{d_{(s-1)s}} & -a_{s} 
 	\end{matrix}\right), \qquad s = 1,\ldots, g-2. \]
By, for example, an easy induction argument, one can compute that \[\det(B'_s) = \frac{(-1)^sN_{s+1}}{N_1\prod_{l=1}^sd_{l(l+1)}}.\] Since $N_{s+1} = n_{s+1}\lbeta_{s+1}$ and $N_1 = \prod_{l=0}^gn_l$, we can conclude. \\

Checking the semigroup condition and finding the splice type equations are now very easy. Denote by $w$ for $k = 0,\ldots, g$ the leaf corresponding to $n_w$, and relate to $w$ the variable $z_w$. For the edge to the right of $\E_k$ for $k = 1,\ldots, g-2$, the numbers $l'_{kw}$ for ${w = k+1,\ldots, g}$ are given by $e_w\prod_{l=k+1}^{w-1}n_l$. Hence, $e_k = n_w l'_{kw}$ for every $w = k+1,\ldots, g$ or, thus, the edge weight $e_k$ is indeed contained in the semigroup $\langle l'_{kw} \mid w = k+1,\ldots, g\rangle $. For the edge to the left of $\E_k$ for $k = 2,\ldots, g-1$, the numbers $l'_{kw}$ for $w = 0,\ldots, k-1$ are given by $\frac{\lbeta_w}{e_w}\prod_{l=w+1}^{k-1}n_l = \frac{\lbeta_w}{e_{k-1}}$ so that $\frac{\lbeta_k}{e_k} = \frac{n_k\lbeta_k}{e_{k-1}} = b_{k0}\frac{\lbeta_0}{e_{k-1}}+\cdots +b_{k(k-1)}\frac{\lbeta_{k-1}}{e_{k-1}}$ is contained in the semigroup $\langle l'_{kw} \mid w = 0,\ldots, k-1\rangle $. It follows that the semigroup condition is fulfilled. Furthermore, along the same lines, we have shown that the following equations are of strict splice type: 
\[\renewcommand{\arraystretch}{1.2}{\left\{\begin{array}{l c l c l l}
	z_0^{n_0} & + &  z_1^{n_1} & + & z_2^{n_2} &  = 0 \\
	z_2^{n_2} & + & z_3^{n_3} & + & z_0^{b_{20}}z_1^{b_{21}} &= 0 \\
	& & & & & \vdots  \\
	z_{g-1}^{n_{g-1}} & + & z_g^{n_g} & + & z_0^{b_{(g-1)0}}z_1^{b_{(g-1)1}}\cdots z_{g-2}^{b_{(g-1)(g-2)}} & = 0.
	\end{array}\right.}\] 
These are the equations of $(S,0)$ up to higher order terms and coefficients. In other words, the singularity $(S,0)$ is indeed of splice type.


\end{document}